\documentclass[a4paper]{amsart}
\usepackage[utf8]{inputenc}
\usepackage{amsthm, amsmath, amsfonts, amssymb, commath, url, hyperref, cleveref, xcolor, tikz, tikz-cd}
\usepackage[scr=boondox]{mathalfa}
\usepackage{enumitem}
\usepackage[backend=biber,style=alphabetic,doi=false]{biblatex}
\setlist[enumerate,1]{label={(\roman*)}}

\numberwithin{equation}{section}
\theoremstyle{definition}
\newtheorem{defn}{Definition}[section]
\newtheorem{example}[defn]{Example}

\theoremstyle{remark}
\newtheorem{rmk}[defn]{Remark}
\theoremstyle{plain}
\newtheorem{lem}[defn]{Lemma}
\newtheorem{prop}[defn]{Proposition}
\newtheorem{thm}[defn]{Theorem}

\newcommand{\Z}{\mathrm{Z}}

\newcommand{\X}{\mathrm{X}}
\newcommand{\T}{\mathrm{T}}

\DeclareMathAlphabet{\bb}{U}{bbold}{m}{n}

\DeclareMathOperator{\id}{id}
\DeclareMathOperator{\supp}{supp}
\newcommand{\nat}{\mathbb{N}}
\newcommand{\real}{\mathbb{R}}

\DeclareMathOperator{\pr}{pr}

\DeclareMathOperator{\cat}{cat}
\DeclareMathOperator{\loc}{loc}

\newcommand{\orbifold}[1]{\mathscr{#1}}

\DeclareMathOperator{\reg}{reg}
\DeclareMathOperator{\sing}{sing}

\newcommand{\rad}{\mathrm{rad}}

\newcommand{\B}{\mathrm{B}}
\DeclareMathOperator{\length}{length}
\DeclareMathOperator{\dist}{dist}

\newcommand{\metric}[1]{\mathfrak{#1}}

\DeclareMathOperator{\Cm}{Cm}
\DeclareMathOperator{\cm}{cm}

\DeclareMathOperator{\grad}{grad}

\DeclareMathOperator{\inj}{inj}

\DeclareMathOperator{\GL}{GL}

\DeclareMathOperator{\CC}{CC}
\DeclareMathOperator{\CF}{CF}

\title[Solutions to a Nonlinear Problem on an Orbifold]{Multiplicity of Solutions to a Nonlinear Elliptic Problem on a Riemannian Orbifold}
\author{Gustavo de Paula Ramos}
\address{Instituto de Matem\'atica e Estat\'istica\\
Universidade de S\~ao Paulo\\
Rua do Mat\~ao, 1010\\
05508-090\\
S\~ao Paulo, SP\\
Brazil}
\email{gustavopramos@gmail.com, gpramos@ime.usp.br}
\urladdr{http://gpramos.com}
\begin{document}

\begin{abstract}
We employ the photography method to obtain a lower bound for the number of solutions to a nonlinear elliptic problem on a Riemannian orbifold in function of the Lusternik--Schnirelmann category of its submanifold of points with largest local group.

\end{abstract}

\date{\today}
\maketitle

\section{Introduction}
\subsection*{Main result}\label{Context}

Let $(\orbifold{O},\metric{g})$ be an $n$-dimensional orientable compact connected Riemannian orbifold, where $n\geq 3$. We are interested in the multiplicity of weak solutions to the subcritical nonlinear problem
\begin{equation}\label{Introduction:Eqn:Main}\tag{$P_\epsilon$}
	\begin{cases}
		-\epsilon^2\Delta^{(\orbifold{O},\metric{g})}u+u=u|u|^{p-2}~\text{and}\\
		u>0
	\end{cases}
\end{equation}
when $\epsilon\in]0,1[$ is sufficiently small, where $2<p<2n/(n-2)$.

It is already known (\cite[Theorem A]{Benci2007}) that the analogous problem on a (not necessarily orientable) compact connected Riemannian manifold $M$ admits at least $\cat(M)+1$ different non-constant solutions, where $\cat$ denotes the Lusternik--Schnirelmann category (see Definition \ref{PhotographyMethod:Defn}).

Inspired by this theorem, our main result is that we can estimate the number of solutions to \eqref{Introduction:Eqn:Main} with respect to the topology of
\[
	\Z^{\orbifold{O}}:=\left\{
		x\in\orbifold{O}: |\Gamma^{\orbifold{O}}(x)|=\max_{y\in\orbifold{O}}|\Gamma^{\orbifold{O}}(y)|=:\zeta^{\orbifold{O}}
	\right\},
\]
where $\Gamma^\orbifold{O}(x)$ denotes the local group of $x\in\orbifold{O}$ (see Definition \ref{Suborbifolds:Defn:LocalGroup}). We remark that this estimate is intimately related to the canonical way which one integrates on $(\orbifold{O},\metric{g})$ (see Definition \ref{RiemannianStructures:Defn:Radon}). Let us state our main result more precisely.

\begin{thm}\label{Introduction:Theorem:A}
If $\epsilon\in ]0,1[$ is sufficiently small, then \eqref{Introduction:Eqn:Main} admits at least $\cat(\Z^{\orbifold{O}})+1$ different non-constant weak solutions.
\end{thm}

As the orbifold $\orbifold{O}$ is a manifold precisely when $\Z^{\orbifold{O}}=\orbifold{O}$, Theorem \ref{Introduction:Theorem:A} implies the weaker form of \cite[Theorem A]{Benci2007} which only accounts for orientable manifolds. In fact, the orientability of $\orbifold{O}$ is only used to employ Sobolev embeddings and thus the result may be easily extended to non-orientable orbifolds if the Sobolev embeddings are proved to hold in this situation.

\subsection*{Variational framework}

If $1\leq q<\infty$, then we define the Lebesgue space $L^q(\orbifold{O},\metric{g})$ as the completion of $C^\infty(\orbifold{O})$ with respect to the norm
\[
	\|u\|_{L^q(\orbifold{O},\metric{g})}:=
	\left[
		\int_\orbifold{O} |u|^q\mathrm{d}\mu^{(\orbifold{O},\metric{g})}
	\right]^{1/q},
\]
where we define the space of real-valued orbifold maps $C^\infty(\orbifold{O})$ on Section \ref{Orbifolds}. Analogously, the Sobolev space $H^1(\orbifold{O},\metric{g})$ is defined as the Hilbert space obtained as completion of $C^\infty(\orbifold{O})$ with respect to the inner product
\[
	\langle u,v\rangle_{H^1(\orbifold{O},\metric{g})}:=
	\int_\orbifold{O}
		\metric{g}(\grad^{(\orbifold{O},\metric{g})}(u),\grad^{(\orbifold{O},\metric{g})}(v))+uv
	\mathrm{d}\mu^{(\orbifold{O},\metric{g})}.
\]

The Sobolev embeddings are known to hold in this context (see \cite[Theorem 2.3]{Farsi2001}), so we define a \emph{weak solution} to \eqref{Introduction:Eqn:Main} as being a $u\in H^1(\orbifold{O},\metric{g})$ such that
\begin{multline*}
	\text{given}~v\in C^\infty(\orbifold{O}),
	\\
	\epsilon^2\langle
		u,v
	\rangle_{H^1(\orbifold{O},\metric{g})}
	+
	(1-\epsilon^2)\langle
		u,v
	\rangle_{L^2(\orbifold{O},\metric{g})}
	+
	\langle
		i^*(u(u^+)^{p-2}),v
	\rangle_{H^1(\orbifold{O},\metric{g})}
	=
	0,
\end{multline*}
where $i\colon H^1(\orbifold{O},\metric{g})\to L^p(\orbifold{O},\metric{g})$ is a Sobolev embedding. As usual in the literature, we will adopt the abuse of language of writing the condition above as
\begin{equation}\label{Introduction:WeakSolution:1}
	\text{given}~v\in C^\infty(\orbifold{O}),~
	\int_{\orbifold{O}}
		\epsilon^2\metric{g}(\grad^{(\orbifold{O},\metric{g})}(u),\grad^{(\orbifold{O},\metric{g})}(v))+uv-u(u^+)^{p-2}v
	\mathrm{d}\mu^{(\orbifold{O},\metric{g})}=0.
\end{equation}
The divergence theorem holds in our context (see \cite[p. 320-321]{Chiang1990}), so our notion of weak solutions is akin to the respective concept on Riemannian manifolds (without boundary). In fact, we have a well-defined functional $J_\epsilon\colon H^1(\orbifold{O},\metric{g})\to\real$ given by
\begin{equation}\label{Intro:Eqn:Functional}
	J_\epsilon(u):=\frac{1}{\epsilon^n}\int_\orbifold{O}\frac{\epsilon^2}{2}\metric{g}(\grad^{(\orbifold{O},\metric{g})}(u),\grad^{(\orbifold{O},\metric{g})}(u))+\frac{u^2}{2}-\frac{(u^+)^p}{p}\mathrm{d}\mu^{(\orbifold{O},\metric{g})}
\end{equation}
which can be routinely shown to be in $C^2(H^1(\orbifold{O},\metric{g}))$. It is therefore possible to deduce the variational characterization that $u\in H^1(\orbifold{O},\metric{g})$ satisfies \eqref{Introduction:WeakSolution:1} if, and only if, $\mathrm{d}_uJ_\epsilon=0$.

We can also use the Sobolev embeddings to prove that
\[
	m(J_\epsilon):=\inf_{\mathcal{N}_\epsilon}J_\epsilon>0,
\]
where
\[
	\mathcal{N}_\epsilon:=\{u\in H^1(\orbifold{O},\metric{g})\setminus\{0\}: \mathrm{d}_uJ_\epsilon(u)=0\}
\]
is the \emph{Nehari manifold} associated to $J_\epsilon$, which is a natural constraint of $J_\epsilon$ (see \cite[Section 6.4]{Ambrosetti2007}). We also remark that the Rellich--Kondrakov theorem holds in this context (see \cite[Theorem 2.4]{Farsi2001}), so $\left.J_\epsilon\right|_{\mathcal{N}_\epsilon}$ satisfies the Palais--Smale condition. To finish, we fix the notation for the (possibly empty) intersection between $\mathcal{N}_\epsilon$ and the sublevels of $J_\epsilon$,
\[
	\Lambda_{\epsilon,\alpha}:=\{
		u\in\mathcal{N}_\epsilon: J_\epsilon(u)<\alpha
	\}
\]
for any $\alpha>0$.

\subsection*{Sketching the proof of Theorem \ref{Introduction:Theorem:A}}\label{PhotographyMethod}
We begin by recalling the definition of the Lusternik--Schnirelmann category.
\begin{defn}\label{PhotographyMethod:Defn}
Let $X$ be a topological space. If $\emptyset\neq A\subset X$, then we define the \emph{Lusternik--Schnirelmann category} of $A$ as a subset of $X$ as
\begin{multline*}
	\cat_X(A)=\min\left\{
		n\in\nat: A~\text{is contained in the union of}
	\right.
	\\
	\left.
		n~\text{contractible closed subsets of}~X
	\right\}
\end{multline*}
and we set $\cat_X(\emptyset)=0$. Moreover, we define $\cat(X):=\cat_X(X)$.
\end{defn}

The following well-known result shows that the Lusternik--Schnirelmann category may be used to estimate the multiplicity of critical points.

\begin{thm}[{\cite[Theorem 3.3]{Benci2007}}]\label{PhotographyMethod:Theorem}
Let $\mathcal{N}$ be a $C^{1,1}$-Banach manifold and let $J\in C^1(\mathcal{N})$. If $J$ is bounded below and satisfies the Palais--Smale condition, then it admits at least $\cat(J^d)$ critical points in $J^d:=\{u\in\mathcal{N}:J(u)<d\}$. Moreover, if $\mathcal{N}$ is contractible and $\cat(J^d)>1$, then $J$ has at least one critical point outside of $J^d$.
\end{thm}

In the context of the previous theorem, the \emph{photography method} is a technique (thoroughly described in \cite[Section 4]{Benci2022}) that allows us to estimate the Lusternik--Schnirelmann category of $J^d$ when $\mathcal{N}$ is a submanifold of a Sobolev space.

Back to our problem, we can actually argue as in \cite[Section 6]{Benci2007} to prove that $\mathcal{N}_\epsilon$ is contractible, so we only have to prove that $\cat(\Lambda_{\epsilon,m(J_\epsilon)+\delta})\geq\cat(\Z^{\orbifold{O}})$ to deduce Theorem \ref{Introduction:Theorem:A} from Theorem \ref{PhotographyMethod:Theorem}. The lemma that follows is the key result which we will use to obtain such an estimate.
\begin{lem}[{\cite[Remark 2.2]{Benci2007}}]
Let $X,Y$ be topological spaces and let $f\colon X\to Y$, $g\colon Y\to X$ be continuous maps such that $g\circ f\simeq\id_X$. We conclude that $\cat(X)\leq\cat(Y)$.
\end{lem}

The goal of the paper then becomes proving the following result.

\begin{thm}\label{Introduction:Theorem:ToProve}
If $\epsilon\in]0,1[$ is sufficiently small, then there exists $\delta>0$ for which we can construct continuous maps $\Z^{\orbifold{O}}\to\Lambda_{\epsilon,m(J_\epsilon)+\delta}$ and $\Lambda_{\epsilon,m(J_\epsilon)+\delta}\to\orbifold{O}$ whose composition is homotopic to $\id_{\Z^{\orbifold{O}}}$.
\end{thm}

\subsection*{Organization of the paper}

In Section \ref{Orbifolds}, we develop background material on orbifolds by introducing basic definitions, fixing notation and recalling relevant results. Similarly to Section \ref{Orbifolds}, most of Section \ref{RiemannianOrbifolds} is dedicated to basic material on Riemannian orbifolds yet we finish this section by introducing the concept of \emph{weak convexity} and proving that the \emph{Riemannian center of mass} may be considered in this context. In Section \ref{Preliminaries}, we develop the technical preliminaries to prove Theorem \ref{Introduction:Theorem:ToProve} in Section \ref{ProofA}.

\subsection*{Notation and terminology}

If $(X,\dist)$ is a metric space or has a canonically induced metric space structure, as in the case of inner product spaces or Riemannian manifolds, then we define
\[
	\B^{(X,\dist)}(Y,r)=\left\{x\in X:\inf_{y\in Y}\dist(x,y)<r\right\}
\]
for any $Y\subset X$ and $r\in]0,\infty[$. If $(M,\metric{g})$ is a Riemannian manifold, then we say that $A\subset M$ is \emph{strongly convex} on $(M,\metric{g})$ when given $x,y\in A$,
\begin{enumerate}
\item
	up to parametrization, there exists a unique minimizing geodesic on $(M,\metric{g})$ which links $x$ to $y$ and
\item
	if $\gamma\colon [0,1]\to (M,\metric{g})$ is a minimizing geodesic which links $x$ to $y$, then $\gamma([0,1])\subset A$.
\end{enumerate}

\subsection*{Acknowledgments}

This study was financed in part by the Coordenação de Aperfeiçoamento de Pessoal de Nível Superior - Brasil (CAPES) - Finance Code 001. More precisely, by CAPES grants 88887.614697/2021-00 and 88887.715990/2022-00. The author expresses his gratitude to Francisco Caramello and André Magalhães de Sá Gomes for very helpful conversations about (Riemannian) orbifolds.

\section{Orbifolds}\label{Orbifolds}

In this section, we present (smooth) orbifolds by following the classical approach via local charts and atlases, similarly as in \cite{Caramello2019,Alfonso2021}.

\subsection*{The orbifold structure}

Let $X$ be a topological space.

\begin{defn}
Given $x\in X,$ we call a triple $(\widetilde{U},G,\phi)$ an $n$-dimensional \emph{orbifold chart} around $x$ on $X$ when
\begin{enumerate}
\item
	$\widetilde{U}$ is a connected open subset of $\real^n,$ called the \emph{domain} of the chart;
\item
	$G$ is a finite group acting smoothly and effectively on $\widetilde{U}$, called the \emph{structural group} of the chart;
\item \label{Orbifolds:Defn:Invariant}
	the \emph{chart map} $\phi\colon\widetilde{U}\to X$ is a continuous $G$-invariant map which induces a homeomorphism $\widetilde{U}/G\to\phi(\widetilde{U})$ and
\item
	$\phi^{-1}(x)$ is unitary.
\end{enumerate}
\end{defn}

Suppose that $(\widetilde{U},G,\phi)$ is an orbifold chart on $X$. In this situation, we say that $U:=\phi(\widetilde{U})\subset X$ is a \textit{fundamental open subset} of $X$
and $\Phi_g\colon\widetilde{U}\to\widetilde{U}$ will denote the diffeomorphism induced by the action of $g\in G$.

\begin{defn}
Suppose that $(\widetilde{U},G,\phi), (\widetilde{V},H,\psi)$ are $n$-dimensional orbifold charts on $X$. If $\lambda\colon \widetilde{U}\to \widetilde{V}$ is a smooth embedding and $\psi\circ\lambda=\phi$, then $\lambda$ is said to be an \emph{orbifold embedding} of $(\widetilde{U},G,\phi)$ into $(\widetilde{V},H,\psi)$.
\end{defn}

We introduce the concept of orbifold atlases.
\begin{defn}
The set $\mathcal{A}=\{(\widetilde{U_i},G_i,\phi_i)\}_{i\in I}$ is said to be an $n$-dimensional \emph{orbifold atlas} of $X$ when
\begin{enumerate}
\item
	given $i\in I$, $(\widetilde{U_i},G_i,\phi_i)$ is an $n$-dimensional orbifold chart on $X$;
\item
	$\{U_i\}_{i\in I}$ is an open cover of $X$;
\item
	if $i,j\in I$ are such that $U_i\subset U_j$, then there exists an orbifold embedding of $(\widetilde{U_i},G_i,\phi_i)$ into $(\widetilde{U_j},G_j,\phi_j)$ and
\item
	given $x\in X$ and $i,j\in I$ for which $x \in U_i\cap U_j$, there exists $k\in I$ such that $U_k\subset U_i\cap U_j$.
\end{enumerate}
\end{defn}

If $\mathcal{A}_1, \mathcal{A}_2$ are $n$-dimensional orbifold atlases of $X$, then they are said to be \emph{compatible} when $\mathcal{A}_1\cup\mathcal{A}_2$ is contained in an $n$-dimensional orbifold atlas of $X$.
\begin{defn}
We say that $\widetilde{\mathcal{A}}$ is an $n$-dimensional \emph{orbifold structure} of $X$ when
\begin{enumerate}
\item
	$\widetilde{\mathcal{A}}$ is an $n$-dimensional orbifold atlas of $X$ and
\item
	if $\mathcal{A}$ is an $n$-dimensional orbifold atlas of $X$ and $\mathcal{A}$ is compatible with $\widetilde{\mathcal{A}}$, then $\mathcal{A}\subset\widetilde{\mathcal{A}}$.
\end{enumerate}
\end{defn}

Finally, let us define the notion of an orbifold.
\begin{defn}
The pair $\orbifold{O}=(\X^{\orbifold{O}},\mathcal{A}^\orbifold{O})$ is said to be an $n$-dimensional \emph{orbifold} when $\X^{\orbifold{O}}$ is a paracompact Hausdorff space and $\mathcal{A}^{\orbifold{O}}$ is an $n$-dimensional orbifold structure of $\X^{\orbifold{O}}$.
\end{defn}

Let $\orbifold{O}$ be an orbifold. In this context,
\begin{itemize}
\item
	$\mathcal{B}$ is said to be an \emph{orbifold atlas} of $\orbifold{O}$ when it is an orbifold atlas of $\X^\orbifold{O}$ and $\mathcal{B}\subset\mathcal{A}^{\orbifold{O}}$;
\item
	$(\widetilde{U},G,\phi)$ is said to be an \emph{orbifold chart} on $\orbifold{O}$ when $(\widetilde{U},G,\phi)\in\mathcal{A}^{\orbifold{O}}$ and
\item
	we write $x\in\orbifold{O}$ (resp., $U\subset\orbifold{O}$) to mean that $x\in \X^{\orbifold{O}}$ (resp., $U\subset \X^{\orbifold{O}}$).
\end{itemize}

We finish this section by introducing the notion of orientability of an orbifold.
\begin{defn}
The orbifold $\orbifold{O}$ is said to be \emph{orientable} when we can fix $\mathcal{B}$, an orbifold atlas of $\orbifold{O}$ such that
\begin{enumerate}
\item
	if $(\widetilde{U},G,\phi)\in\mathcal{B}$, then $G$ acts on $\widetilde{U}$ by orientation-preserving diffeomorphisms and
\item
	the orbifolds embeddings between charts of $\mathcal{B}$ are orientation-preserving.
\end{enumerate}	
\end{defn}

\subsection*{Orbifold maps}

Let $\orbifold{O},\orbifold{P}$ be orbifolds and let $f\colon \X^\orbifold{O}\to \X^\orbifold{P}$ be a continuous map.

\begin{defn}\label{OrbifoldMaps:Defn:Lift}
Given $x\in\orbifold{O}$, a smooth function $\widetilde{f_x}\colon\widetilde{U_x}\to\widetilde{V_{f(x)}}$ is said to be a \emph{local lift} of $f$ at $x$ when
\begin{enumerate}
\item
	$(\widetilde{U_x},G_x,\phi_x)$ is an orbifold chart around $x$ on $\orbifold{O}$;
\item
	$(\widetilde{V_{f(x)}},H_{f(x)},\psi_{f(x)})$ is an orbifold chart around $f(x)$ on $\orbifold{P}$;
\item
	$f(U_x)\subset V_{f(x)}$;
\item
	$\psi_{f(x)}\circ\widetilde{f_x}=f\circ\phi_x$ and
\item
	given $g\in G_x$, there exists $h_g\in H_{f(x)}$ such that $\widetilde{f_x}\circ\Phi_g=\Phi_{h_g}\circ\widetilde{f_x}$.
\end{enumerate}
\end{defn}

There exists a notion of isomorphism between the local lifts of $f$ at an $x\in\orbifold{O}$.

\begin{defn}
Fix $x\in\orbifold{O}$ and suppose that $\widetilde{f^1_x}\colon\widetilde{U^1_x}\to\widetilde{V^1_{f(x)}}$, $\widetilde{f^2_x}\colon\widetilde{U^2_x}\to\widetilde{V^2_{f(x)}}$ are local lifts of $f$ at $x$. We say that $\widetilde{f^1_x}$ and $\widetilde{f^2_x}$ are isomorphic when there exist smooth diffeomorphisms $\widetilde{\psi^\orbifold{O}}\colon\widetilde{U^1_x}\to\widetilde{U^2_x}$, $\widetilde{\psi^\orbifold{P}}\colon\widetilde{V^1_{f(x)}}\to\widetilde{V^2_{f(x)}}$ such that $\widetilde{\psi^\orbifold{P}}\circ\widetilde{f^1_x}=\widetilde{f^2_x}\circ\widetilde{\psi^\orbifold{O}}$.
\end{defn}

This notion of isomorphism between local lifts induces a concept of equivalence of local lifts.

\begin{defn}
Suppose that $\widetilde{f^1}\colon\widetilde{U^1}\to\widetilde{V^1}$, $\widetilde{f^2}\colon\widetilde{U^2}\to\widetilde{V^2}$ are local lifts of $f$. Given $x\in U^1\cap U^2$, we say that $\widetilde{f^1}$ is equivalent to $\widetilde{f^2}$ at $x$, denoted as $\widetilde{f^1}\sim_x\widetilde{f^2}$, when there exists $(\widetilde{U^3},G^3,\phi^3)$, an orbifold chart around $x$ on $\orbifold{O}$ such that $U^3\subset U^1\cap U^2$ and the lifts $\widetilde{f^1},\widetilde{f^2}$ induce isomorphic lifts on $\widetilde{U^3}$.
\end{defn}

Let $\orbifold{C}^\infty(\orbifold{O},\orbifold{P})$ be the set of elements $(f,\{\widetilde{f_x}\}_{x\in\orbifold{O}})$ where $f\colon\X^{\orbifold{O}}\to\X^{\orbifold{P}}$ is continuous and given $x\in\orbifold{O}$, $\widetilde{f_x}$ is a local lift of $f$ at $x$. We introduce an equivalence relation at $\orbifold{C}^\infty(\orbifold{O},\orbifold{P})$ by setting $(f^1,\{\widetilde{f^1_x}\}_{x\in\orbifold{O}})\sim(f^2,\{\widetilde{f^2_x}\}_{x\in\orbifold{O}})$ when given $x\in\orbifold{O}$, $\widetilde{f^1_y}\sim_x\widetilde{f^2_z}$ for any $y,z\in\orbifold{O}$.

\begin{defn}
We say that $\orbifold{f}$ is an \emph{orbifold map} from $\orbifold{O}$ to $\orbifold{P}$, denoted as $\orbifold{f}\colon\orbifold{O}\to\orbifold{P}$, when $\orbifold{f}\in C^\infty(\orbifold{O},\orbifold{P}):=\orbifold{C}^\infty(\orbifold{O},\orbifold{P})/{\sim}$.
\end{defn}

As we are particularly interested in real-valued orbifold maps, we set $C^\infty(\orbifold{O})=C^\infty(\orbifold{O},\real)$.

\subsection*{Orbibundles and the tangent orbibundle}\label{Orbibundles}
Let $\orbifold{P}$ be an \emph{orbibundle} over the orbifold $\orbifold{O}$, i.e.,
\begin{enumerate}
\item
	$\orbifold{P}$ is an orbifold;
\item \label{Orbibundles:Trivialization}
	given $x\in\orbifold{O}$, there exist
	\begin{enumerate}
	\item
		$(\widetilde{U_x},G_x,\phi_x)$ which is an orbifold chart around $x$ on $\orbifold{O}$;
	\item
		$(\widetilde{V_x},G_x,\psi_x)$ which is an orbifold chart on $\orbifold{P}$;
	\item
		a smooth bundle map $\widetilde{\pi_x}\colon\widetilde{V_x}\to\widetilde{U_x}$ and
	\end{enumerate}
\item \label{Orbibundles:Embeddings}
	if
	\begin{enumerate}
	\item
		$(\widetilde{U_1},G_1,\phi_1)$ and $(\widetilde{U_2},G_2,\phi_2)$ are orbifold charts on $\orbifold{O}$;
	\item
		$(\widetilde{V_1},G_1,\psi_1)$ and $(\widetilde{V_2},G_2,\psi_2)$ are orbifold charts on $\orbifold{P}$;
	\item
		$\widetilde{\pi_1}\colon\widetilde{V_1}\to\widetilde{U_1}$ and $\widetilde{\pi_2}\colon\widetilde{V_2}\to\widetilde{U_2}$ are smooth bundle maps and
	\item
		$\lambda^{\orbifold{O}}\colon\widetilde{U_1}\to\widetilde{U_2}$ is an orbifold embedding of $(\widetilde{U_1},G_1,\phi_1)$ into $(\widetilde{U_2},G_2,\phi_2)$,
	\end{enumerate}
	then there exists $\lambda^{\orbifold{P}}\colon\widetilde{V_1}\to\widetilde{V_2}$, an orbifold embedding of $(\widetilde{V_1},H_1,\psi_1)$ into $(\widetilde{V_2},H_2,\psi_2)$ such that $\widetilde{\pi_2}\circ\lambda^{\orbifold{P}}=\lambda^{\orbifold{O}}\circ\widetilde{\pi_1}$.
\end{enumerate}

The orbibundle $\orbifold{P}$ is said to be a \emph{cone orbibundle} of rank $k$ when the bundle maps in \ref{Orbibundles:Trivialization} define vector bundles of rank $k$ and the mappings in \ref{Orbibundles:Embeddings} are vector bundle morphisms. Furthermore, the orbibundle structure on $\orbifold{P}$ canonically induces the unique surjective orbifold map $\pi^{\orbifold{P}}\colon\orbifold{P}\to\orbifold{O}$ such that given $x\in\orbifold{O}$ and the respective information in \ref{Orbibundles:Trivialization}, it holds that $\pi^{\orbifold{P}}\circ\psi_x=\phi_x\circ\widetilde{\pi_x}$. We call $\sigma$ an \emph{orbisection} of $\orbifold{P}$ when $\sigma\in C^\infty(\orbifold{O},\orbifold{P})$ and $\pi^{\orbifold{P}}\circ\sigma=\id^{\orbifold{O}}$. To finish the discussion about orbibundles, we define the \emph{fiber} of $\orbifold{P}$ over $x\in\orbifold{O}$ as the set $\orbifold{P}_x:=(\pi^{\orbifold{P}})^{-1}(x)\subset\orbifold{P}$ and given $U\subset\orbifold{O}$, we define $\left.\orbifold{P}\right|_U=(\pi^{\orbifold{P}})^{-1}(U)$.

Let us construct the simplest and probably most important cone orbibundle over $\orbifold{O}$, the tangent orbibundle $\T\orbifold{O}$. Suppose that $\orbifold{O}$ is $n$-dimensional and let $\mathcal{B}=\{(\widetilde{U_i},G_i,\phi_i)\}_{i\in I}$ be an orbifold atlas of $\orbifold{O}$. Given $i\in I$, let $\widetilde{\pi_i}=\pr_1\colon \widetilde{U_i}\times\real^n\to\widetilde{U_i}$; consider the action of $G_i$ on $\widetilde{V_i}:=\widetilde{U_i}\times\real^n$ given by
\[
	g(\widetilde{x},v)
	:=
	(g\widetilde{x},\mathrm{d}_{\widetilde{x}}\Phi_g(v))\in\widetilde{U_i}\times\real^n
\]
for every $(g,\widetilde{x},v)\in G_i\times\widetilde{U_i}\times\real^n$ and define the quotient topological space $Y_i=\widetilde{V_i}/G_i$. Now, define the product topological space $Y^{\T\orbifold{O}}=\prod_{i\in I}(\{i\}\times Y_i)$. Given $i,j\in I$; $(\widetilde{x_i},v_i)\in\widetilde{V_i}$ and $(\widetilde{x_j},v_j)\in\widetilde{V_j}$, we set $(i, G_i(\widetilde{x_i},v_i))\sim (j, G_j(\widetilde{x_j},v_j))$ precisely when
\begin{enumerate}
\item
	$\phi_i(\widetilde{x_i})=\phi_j(\widetilde{x_j})=:y\in U_i\cap U_j$;
\item
	we can fix $\lambda\colon\widetilde{U_i}\to\widetilde{U_j}$, an orbifold embedding of $(\widetilde{U_i},G_i,\phi_i)$ into $(\widetilde{U_j},G_j,\phi_j)$ such that $\mathrm{d}_{\widetilde{x_i}}\lambda(v_i)=v_j$.
\end{enumerate}
We set the quotient topological space $\X^{\T\orbifold{O}}=Y^{\T\orbifold{O}}/{\sim}$. The last step to obtain the tangent orbibundle $\T\orbifold{O}$ consists of defining an orbifold atlas for $\X^{\T\orbifold{O}}$. Let $\mathrm{d}\phi_i\colon\widetilde{V_i}\to \X^{\T\orbifold{O}}$ be defined as $\mathrm{d}\phi_i(\widetilde{x},v)=[i,G_i(\widetilde{x},v)]$ for any $(\widetilde{x},v)\in\widetilde{V_i}$ and $i\in I$.
\begin{defn}
The \emph{tangent orbibundle} of $\orbifold{O}$ is the cone orbibundle of rank $n$ over $\orbifold{O}$ given by $\T\orbifold{O}=(\X^{\T\orbifold{O}},\mathcal{A}^{\T\orbifold{O}})$, where $\mathcal{A}^{\T\orbifold{O}}$ is generated by the orbifold atlas
$
	\{(\widetilde{V_i},G_i,\mathrm{d}\phi_i)\}_{i\in I}.
$
\end{defn}

\subsection*{Suborbifolds and the canonical stratification}\label{Suborbifolds}
Let $\orbifold{O}$ be an $n$-dimensional orbifold. We begin by introducing a particular kind of orbifold chart on $\orbifold{O}$.
\begin{lem}
Given $x\in\orbifold{O}$, there exists a \emph{linear chart} centered at $x$ on $\orbifold{O}$, i.e., an $(\real^n,G,\phi)$ which is an orbifold chart around $x$ on $\orbifold{O}$ such that $\phi^{-1}(x)=0$ and $G$ is a subgroup of $\GL(\real^n)$ acting canonically on $\real^n$.
\end{lem}
\begin{proof}
Let $(\widetilde{V},H,\psi)$ be an orbifold chart around $x$ on $\orbifold{O}$ and set $\widetilde{x}=\phi^{-1}(x).$ Let $\widetilde{\metric{g}}$ be a $G$-equivariant Riemannian metric on $\widetilde{V}$ and fix $\epsilon>0$ such that 
\[
	\real^n\supset B^{(\real^n,\widetilde{\metric{g}}_{\widetilde{x}})}(0,\epsilon)\ni u \mapsto \exp^{(\widetilde{V},\widetilde{\metric{g}})}_{\widetilde{x}}(u)\in B^{(\widetilde{V},\widetilde{\metric{g}})}(\widetilde{x},\epsilon)\subset\widetilde{V}
\]
is a smooth diffeomorphism. In particular, we can define the following smooth diffeomorphism:
\[
	\real^n \ni u\mapsto \eta(u):=\exp^{(\widetilde{V},\widetilde{\metric{g}})}_{\widetilde{x}}\left(
		\frac{\epsilon|u|^2}{1+|u|^2}u
	\right)\in B^{(\widetilde{V},\widetilde{\metric{g}})}(\widetilde{x},\epsilon)\subset\widetilde{V}.
\]
The Riemannian metric $\widetilde{\metric{g}}$ is $G$-equivariant, so
\[
	\exp^{(\widetilde{V},\widetilde{\metric{g}})}_{\widetilde{x}}\circ\mathrm{d}_{\widetilde{x}}\Phi_g=\Phi_g\circ\exp^{(\widetilde{V},\widetilde{\metric{g}})}_{\widetilde{x}}
\]
for any $g\in G$. Finally, $(\real^n,G,\phi:=\psi\circ\eta)$ is an orbifold chart which satisfies the conclusion of the lemma, where $G:=\{\eta^*\mathrm{d}_{\widetilde{x}}\Phi_h: h \in H\}\subset \GL(\real^n)$.
\end{proof}

Having introduced the concept of linear charts, we can define what we mean by embedded suborbifolds of $\orbifold{O}$.

\begin{defn}
We say that $\orbifold{P}$ is a $k$-dimensional \emph{embedded suborbifold} of $\orbifold{O}$ when
\begin{enumerate}
\item
$\orbifold{P}$ is a $k$-dimensional orbifold;
\item
$\X^\orbifold{P}$ is a topological subspace of $\X^\orbifold{O};$
\item
if $x\in\orbifold{P}$, then there exists $(\real^n,G,\phi)$, a linear chart centered at $x$ on $\orbifold{O}$ that is \emph{adapted} to $\orbifold{P}$, i.e., $(\real^k,H,\psi:=\phi\circ\pr_1)$ is an orbifold chart on $\orbifold{P}$ and $\psi(\real^k)=\phi(\real^k\times\{0_{\real^{n-k}}\})=\phi(\real^n)\cap\orbifold{P}$, where $H$ is a certain subgroup of $G$.
\end{enumerate}
\end{defn}

Now, let us present the canonical stratification of $\orbifold{O}$. Recall the following simple result.

\begin{lem}[{\cite[Lemma 2.10]{Moerdijk2003}}]\label{Orbifolds:Lem:Monomorphism}
If $(\widetilde{U},G,\phi)$ is an orbifold chart around $x$ on $\orbifold{O}$, then $G\ni g \mapsto \mathrm{d}_{\widetilde{x}}\Phi_g \in \GL(\real^n)$ is a monomorphism, where $\widetilde{x}:=\phi^{-1}(x)$.
\end{lem}

In the context of the previous lemma, it follows that the subgroup of $\GL(\real^n)$ defined as
\[
	\mathrm{d}_xG=\{\mathrm{d}_{\widetilde{x}}\Phi_g: g\in G\}\subset\GL(\real^n)
\]
is isomorphic to $G$. It follows that the structural groups of orbifold charts around $x\in\orbifold{O}$ are all isomorphic.
\begin{lem}[{see \cite[p. 39--40]{Moerdijk2003}}]\label{Orbifolds:Lem:LocalGroup}
If $x\in\orbifold{O}$ and $(\widetilde{U},G,\phi), (\widetilde{V},H,\psi)$ are orbifold charts around $x$ on $\orbifold{O},$ then 
\[
	\mathrm{d}_xG
	=
	k^{-1}(\mathrm{d}_xH)k
	=
	\{
		k^{-1}hk: h\in H_x
	\}
\]
for a certain $k\in\GL(\real^n).$
\end{lem}

Due to the previous lemma, we can introduce the notion of local groups and strata on an orbifold.
\begin{defn}\label{Suborbifolds:Defn:LocalGroup}
Given $x\in\orbifold{O}$, we define its \emph{local group} as
\[
	\Gamma^{\orbifold{O}}(x)
	=
	\{h^{-1}(\mathrm{d}_xG)h: h\in\GL(\real^n)\}
\]
and we set $|\Gamma^{\orbifold{O}}(x)|:=|G|$, where $G$ is the structural group of any orbifold chart around $x$ on $\orbifold{O}$. The \emph{stratum} $\Sigma^{\orbifold{O}}(x)$ is defined as the connected component of
$
	\{y\in\orbifold{O}: \Gamma^\orbifold{O}(x)=\Gamma^\orbifold{O}(y)\}
$
which contains $x$.
\end{defn}

Let us prove that the strata of $\orbifold{O}$ canonically inherit a manifold structure from the orbifold structure of $\orbifold{O}$.
\begin{lem}\label{Orbifolds:Lem:StratumManifold}
If $\Sigma$ is a stratum of $\orbifold{O},$ then $\Sigma$ admits a natural manifold structure such that $\Sigma$ is an embedded submanifold of $\orbifold{O}$.
\end{lem}
\begin{proof}
Fix $x\in\Sigma$, let $(\real^n,G_x,\phi_x)$ be a linear chart centered at $x$ on $\orbifold{O}$ and set
\[
	V_x=\{u\in\real^n: gu=u~\text{for any}~g\in G_x\}.
\]
As $V_x$ is a vector subspace of $\real^n$, we can fix a linear isomorphism $\Psi_x\colon\real^{k_x}\to V_x$ for a certain $k_x\in\{0,\ldots,n\}$.

We want to show that $k:=k_{x_1}=k_{x_2}$ for any $x_1,x_2\in\Sigma.$ Let $(\real^n,G_{x_1},\phi_{x_1})$ and $(\real^n,G_{x_2},\phi_{x_2})$ be, respectively, linear charts centered at $x_1$ and $x_2$ on $\orbifold{O}$. It is a corollary of the definitions of $V_{x_1}$ and $V_{x_2}$ that these are isomorphic vector spaces, so $k_{x_1}=k_{x_2}.$

We conclude that $\{(\real^k,\phi_x\circ\Psi_x)\}_{x\in\Sigma}$ is a smooth atlas of $\Sigma$ which induces the structure in the statement of the lemma.
\end{proof}

We introduce a decomposition of $\orbifold{O}$ according to the local group of its points. The \emph{regular locus} of $\orbifold{O}$ is defined as
\[
	\orbifold{O}^{\reg}=\{x\in\orbifold{O}: \Gamma^{\orbifold{O}}(x)=\{\id_{\real^n}\}\}
\]
and we call its complement $\orbifold{O}^{\sing}=\orbifold{O}\setminus\orbifold{O}^{\reg}$ the \emph{singular locus} of $\orbifold{O}$. In fact, $\orbifold{O}^{\reg}$ is a dense open subset of $\orbifold{O}$ (see \cite[Lemma 2.10]{Moerdijk2003}) and so it follows from the previous lemma that if $\Sigma$ is a singular stratum of $\orbifold{O}$, then it is a closed submanifold of $\orbifold{O}$.

\section{Riemannian orbifolds}\label{RiemannianOrbifolds}

This section is loosely based on \cite[Section 6]{Borzellino2008} and it aims to introduce the concept of a Riemannian orbifold and usual constructions in this context. As novelty, we introduce the concept of \emph{weakly convex} subsets and we show that the \emph{Riemannian center of mass} may be considered on Riemannian orbifolds.

\subsection*{The Riemannian orbifold structure}

Suppose that $\orbifold{O}$ is an $n$-dimensional orbifold. Given $x\in\orbifold{O},$ we say that $\metric{g}_x\colon \T_x\orbifold{O}\times \T_x\orbifold{O}\to\real$ is an \emph{inner product} when there exist
\begin{enumerate}
\item
	$(\widetilde{U},G,\phi)$, an orbifold chart around $x$ on $\orbifold{O}$, and
\item
	$\widetilde{\metric{g}}_{\widetilde{x}}\colon \real^n\times \real^n\to\real$, a $(G\times G)$-invariant inner product
\end{enumerate}
such that $\widetilde{\metric{g}}_{\widetilde{x}}=(\mathrm{d}_{\widetilde{x}}\phi)^*\metric{g}_x$, where $\widetilde{x}:=\phi^{-1}(x)\in\widetilde{U}$ and $G\times G$ acts on $\real^n\times\real^n$ through
\[
	(g_1,g_2)(v_1,v_2):=(\mathrm{d}_{\widetilde{x}}\Phi_{g_1}(v_1),\mathrm{d}_{\widetilde{x}}\Phi_{g_2}(v_2))
\]
for any $g_1,g_2\in G$ and $u,v\in\real^n$. In this context, $(\orbifold{O},\metric{g})$ is said to be a \emph{Riemannian orbifold} when $\metric{g}=\{\metric{g}_x\}_{x\in\orbifold{O}}$ is a \emph{Riemannian metric} on $\orbifold{O},$ i.e.,
\begin{enumerate}
\item
	given $x\in\orbifold{O},$ $\metric{g}_x\colon T_x\orbifold{O}\times T_x\orbifold{O}\to\real$ is an inner product, and
\item
	if $\sigma,\tau\in\mathfrak{X}(\orbifold{O}),$ then
\[
	\orbifold{O}\ni x\mapsto \metric{g}(\sigma,\tau)(x):=\metric{g}_x(\sigma_x,\tau_x)\in\real
\]
induces an orbifold map.
\end{enumerate}


\subsection*{Riemannian structures on orbifolds}

Let us introduce a few operations on the Riemannian orbifold $(\orbifold{O},\metric{g})$ which are made possible by the presence of a Riemannian metric.

\begin{defn}
If $\phi\colon\widetilde{U}\to\orbifold{O}$ is a chart map, then we canonically associate it to the Riemannian manifold $(\widetilde{U},\widetilde{\metric{g}}:=\phi^*\metric{g})$.
\end{defn}

As in Riemannian manifolds, we can define the gradient of real-valued orbifold maps.

\begin{defn}
Given $u\in C^\infty(\orbifold{O})$, the \emph{gradient} of $u$ on $(\orbifold{O},\metric{g})$ is the vector field $\grad^{(\orbifold{O},\metric{g})}(u)\in\mathfrak{X}(\orbifold{O})$ characterized by the following property: if $\phi\colon\widetilde{U}\to\orbifold{O}$ is a chart map and $\widetilde{u}\colon\widetilde{U}\to\real$ is a local lift of $u$, then
\[
	\widetilde{U}\ni\widetilde{x}
	\mapsto
	(\widetilde{\metric{g}}^{i1}(\widetilde{x})\partial_i\widetilde{u}(\widetilde{x}),\ldots,\widetilde{\metric{g}}^{in}(\widetilde{x})\partial_i\widetilde{u}(\widetilde{x}))
	\in\real^n
\]
is a local lift of $\grad^{(\orbifold{O},\metric{g})}(u)\in\mathfrak{X}(\orbifold{O})$, where the expression above employs Einstein's summation convention.
\end{defn}

The Riemannian metric $\metric{g}$ induces a Radon measure on $\orbifold{O}$.
\begin{defn}\label{RiemannianStructures:Defn:Radon}
We define $\mu^{(\orbifold{O},\metric{g})}$ as the unique measure on $\orbifold{O}$ such that if $(\widetilde{U},G,\phi)$ is an orbifold chart on $\orbifold{O}$ and $B$ is a Borel subset of $U$, then
\[
	\mu^{(\orbifold{O},\metric{g})}(B)=\frac{1}{|G|}\int_{\widetilde{U}}\mathbb{I}_{\phi^{-1}(B)}(\widetilde{x})|\det\widetilde{\metric{g}}(x)|^{1/2}\mathrm{d}\widetilde{x}
\]
whenever $B$ is a Borel subset of $U$.
\end{defn}

\subsection*{The length space structure}

Let us sketch how to canonically associate the connected Riemannian orbifold $(\orbifold{O},\metric{g})$ to a length space $(\orbifold{O},\dist^{(\orbifold{O},\metric{g})})$ (for a reminder about length spaces, we refer the reader to Appendix \ref{LengthSpaces}). We begin by recalling the following result.
\begin{prop}[{\cite[Propositions 36, 37]{Borzellino1992}}]\label{LengthSpace:Proposition:Borzellino}
Suppose that $\phi\colon\widetilde{U}\to U$ is a chart map on $\orbifold{O}$ and $\gamma\colon [a,b]\to U$ is a continuous curve for which there exists a nondecreasing $\{t_k\}_{k\in\nat_0}\subset [a,b]$ such that $t_0=a$, $t_k\to b$ as $k\to\infty$ and given $k\in\nat_0$, $\gamma(]t_k,t_{k+1}[)$ is contained in a single stratum of $\orbifold{O}$. In this situation, there exists $\length^{(\orbifold{O},\metric{g})}(\gamma)\in[0,\infty[$ such that if $\widetilde{\gamma}\colon [a,b]\to\widetilde{U}$ is continuous and $\gamma=\phi\circ\widetilde{\gamma}$, then $\length^{(\widetilde{U},\widetilde{\metric{g}})}(\widetilde{\gamma})=\length^{(\orbifold{O},\metric{g})}(\gamma)$.
\end{prop}

The following definition is inspired by the setting in the previous result.

\begin{defn}[{\cite[Definition 35]{Borzellino1992}}]\label{LengthSpaceStructure:Defn:Admissible2}
A continuous map $\gamma\colon [a',b']\to\orbifold{O}$ is said to be an \emph{admissible curve} whenever the following implication holds: if $U$ is a fundamental open subset of $\orbifold{O}$ and $a,b$ are such that $\gamma([a,b])\subset U$, then $[a',b']\ni t\mapsto \gamma(t)\in U$ satisfies the hypothesis in Proposition \ref{LengthSpace:Proposition:Borzellino}.
\end{defn}

By proceeding as in \cite[Theorem 38]{Borzellino1992}, we define $\length^{(\orbifold{O},\metric{g})}(\gamma)$ whenever $\gamma$ is an admissible curve. Finally, \cite[Theorem 40]{Borzellino1992} shows that the definition
\[
	\dist^{(\orbifold{O},\metric{g})}(x,y)=\inf\{
		\length^{(\orbifold{O},\metric{g})}(\gamma): \gamma~\text{is an admissible curve linking}~x~\text{to}~y
	\}
\]
for every $x,y\in\orbifold{O}$ yields a length space $(\orbifold{O},\dist^{(\orbifold{O},\metric{g})})$.

The concepts of (minimizing) geodesics and (strictly) convex functions are naturally defined in the context of length spaces, so we adopt the convention of considering these concepts on $(\orbifold{O},\metric{g})$ as their respective counterparts on $(\orbifold{O},\dist^{(\orbifold{O},\metric{g})})$. To finish, we remark that the topology induced by $\dist^{(\orbifold{O},\metric{g})}$ matches the topology of the underlying topological space $\X^\orbifold{O}$ and if $(\orbifold{O},\dist^{(\orbifold{O},\metric{g})})$ is a complete metric space, then given $x,y\in\orbifold{O}$, there exists a minimizing geodesic $\gamma\colon [0,1]\to (\orbifold{O},\metric{g})$ linking $x$ to $y$ (see \cite{Gromov1999}). 

\subsection*{The exponential map and the injectivity radius}

Let $(\orbifold{O},\metric{g})$ be a Riemannian orbifold. If $\phi\colon\widetilde{U}\to\orbifold{O}$ is a chart map, then we define
\[
	\widetilde{\Omega}(\widetilde{U},\phi)=\{(\widetilde{x},\widetilde{v})\in\widetilde{U}\times\real^n: |\widetilde{v}|<\inj^{(\widetilde{U},\widetilde{\metric{g}})}(\widetilde{x})\}.
\]
We also set
\[
	\Omega^{(\orbifold{O},\metric{g})}=\bigcup\left\{
		\mathrm{d}\phi(\widetilde{\Omega}(\widetilde{U},\phi)): \phi\colon\widetilde{U}\to\orbifold{O}~\text{is a chart map}
	\right\}\subset\T\orbifold{O}.
\]

\begin{lem}[{\cite[Proposition 6.7]{Borzellino2008}}]
Suppose that $\phi_1\colon\widetilde{U_1}\to\orbifold{O}$, $\phi_2\colon\widetilde{U_2}\to\orbifold{O}$ are chart maps. We conclude that
\[
	\phi_1\circ\exp^{(\widetilde{U_{1}},\widetilde{\metric{g}_{1}})}(\widetilde{x_1},\widetilde{v_1})=\phi_{2}\circ\exp^{(\widetilde{U_{2}},\widetilde{\metric{g}_{2}})}(\widetilde{x_2},\widetilde{v_2})
\]
whenever $(\widetilde{x_1},\widetilde{v_1})\in\widetilde{\Omega}(\widetilde{U_1},\phi_1)$, $(\widetilde{x_2},\widetilde{v_2})\in\widetilde{\Omega}(\widetilde{U_2},\phi_2)$ are such that $\mathrm{d}\phi_1(\widetilde{x_1},\widetilde{v_1})=\mathrm{d}\phi_2(\widetilde{x_2},\widetilde{v_2})\in\Omega^{(\orbifold{O},\metric{g})}$
\end{lem}

The previous lemma assures that the following map is well-defined.

\begin{defn}[The exponential map]
Given $v\in\Omega^{(\orbifold{O},\metric{g})}$, we define
\[
	\exp^{(\orbifold{O},\metric{g})}(v)=\phi\circ\exp^{(\widetilde{U},\widetilde{\metric{g}})}(\widetilde{x},\widetilde{v})
\]
for any triple $(\phi,\widetilde{x},\widetilde{v})$ where $\phi\colon\widetilde{U}\to\orbifold{O}$ is a chart map and $(\widetilde{x},\widetilde{v})\in\widetilde{U}\times\real^n$ is such that $\mathrm{d}\phi(\widetilde{x},\widetilde{v})=v$.
\end{defn}

As in Riemannian manifolds, we define the \emph{injectivity radius} $\inj^{(\orbifold{O},\metric{g})}\colon(\orbifold{O},\metric{g})\to]0,\infty[$ as
\begin{multline}
	\inj^{(\orbifold{O},\metric{g})}(x)
	=
	\sup\{\rho>0:\B^{(\T_x\orbifold{O},\metric{g}_x)}(0,\rho)\subset\Omega^{(\orbifold{O},\metric{g})}\}
	=
	\\
	=
	\sup\left\{
		\inj^{(\widetilde{U},\widetilde{\metric{g}})}(\widetilde{x}) \mid \phi\colon\widetilde{U}\to\orbifold{O}~\text{is a chart map}~\text{and}~\widetilde{x}\in\phi^{-1}(x)\neq\emptyset
	\right\}>0.
\end{multline}

In fact, metric balls on $(\orbifold{O},\metric{g})$ with a sufficiently small radius are fundamental open subsets of $\orbifold{O}$.

\begin{prop}
Suppose that $x\in\orbifold{O}$ and $0<\rho<\inj^{(\orbifold{O},\metric{g})}(x)$. We conclude that we can fix $G\subset \mathrm{O}(\real^n)$ and $\phi\colon\B^{\real^n}(0,\rho)=:\widetilde{U}\to U:=\B^{(\orbifold{O},\metric{g})}(x,\rho)$ for which $(\widetilde{U},G,\phi)$ becomes a \emph{normal chart} centered at $x$ on $(\orbifold{O},\metric{g})$, i.e., there exists a $\metric{g}_x$-orthonormal $\{v_1,\ldots,v_n\}\subset\T_x\orbifold{O}$ such that
\[
	\phi(t)=\exp^{(\orbifold{O},\metric{g})}_x(t_1v_1+\ldots+t_nv_n)
\]
for any $t:=(t_1,\ldots,t_n)\in\widetilde{U}$.
\end{prop}

A similarity with Riemannian manifolds is that radial geodesics are minimizing geodesics.
\begin{lem}
If $v\in\Omega^{(\orbifold{O},\metric{g})}$, then $[0,1]\ni t\mapsto \gamma(t):=\exp^{(\orbifold{O},\metric{g})}(tv)$ is a minimizing geodesic on $(\orbifold{O},\metric{g})$.
\end{lem}

The previous lemma implies a striking difference with Riemannian manifolds. It is well known that singular strata form an obstruction to length minimization (see \cite[Proposition 15]{Borzellino1993}), hence the following result.
\begin{prop}\label{ExponentialMap:Proposition:DiscontinuousInj}
Given $x\in\orbifold{O}^{\reg}$, $\inj^{(\orbifold{O},\metric{g})}(x)\leq\dist^{(\orbifold{O},\metric{g})}(x,\orbifold{O}^{\sing})$. In particular, the inequality $\orbifold{O}^{\sing}\neq\emptyset$ implies $\inj^{(\orbifold{O},\metric{g})}$ discontinuous and $\inf\inj^{(\orbifold{O},\metric{g})}=0$.
\end{prop}

Fortunately, some restrictions of $\inj^{(\orbifold{O},\metric{g})}$ are continuous.
\begin{rmk}\label{ExponentialMap:Remark}
As $\orbifold{O}^{\reg}$ is an open submanifold of $\orbifold{O}$ and $\orbifold{O}^{\sing}$ is a union of closed submanifolds of $\orbifold{O}$, we conclude that $\left.\inj^{(\orbifold{O},\metric{g})}\right|_{\orbifold{O}^{\reg}}$ and $\left.\inj^{(\orbifold{O},\metric{g})}\right|_{\orbifold{O}^{\sing}}$ are continuous.
\end{rmk}

\subsection*{Weak convexity}
It is clasically known that any point on a Riemannian manifold has a strongly convex neighborhood. The following example shows that one cannot hope for the existence of strongly convex neighborhoods around singular points of a Riemannian orbifold $(\orbifold{O},\metric{g})$.

\begin{example}
Let $\orbifold{O}=\real^2/\{\pm\id_{\real^2}\}$; let $\metric{g}$ be the Riemannian metric on $\orbifold{O}$ induced by the Euclidean inner product and let $p\colon\real^2\to\orbifold{O}$ be the canonical projection. We claim that $0\in\orbifold{O}$ does not admit a strongly convex neighborhood on $(\orbifold{O},\metric{g})$. Indeed, given $(x,y)\in\real^2\setminus\{0\}$, the curves on $\real^2$ given for every $t\in[0,1]$ by $\widetilde{\gamma_1}(t)=(1-t)(x,y)+tz(-y,x)$ and $\widetilde{\gamma_2}(t)=(1-t)(x,y)+t(y,-x)$ project to distinct minimizing geodesics on $(\orbifold{O},\metric{g})$ which link $p(x,y)$ to $p(-y,x)$.
\end{example}

Due to the absence of strongly convex neighborhoods around singular points in a Riemannian orbifold, we introduce a weaker notion of convexity.

\begin{defn}
The set $A\subset\orbifold{O}$ is said to be \emph{weakly convex} on $(\orbifold{O},\metric{g})$ when the following implication occurs: if $x,y\in A$ and $\gamma\colon [0,1]\to(\orbifold{O},\metric{g})$ is a minimizing geodesic linking $x$ to $y$, then $\gamma([0,1])\subset A$.
\end{defn}

It turns out that metric balls on $(\orbifold{O},\metric{g})$ are weakly convex if their radius is sufficiently small, which is analogous to the fact that metric balls on Riemannian manifolds are strongly convex if their radius is sufficiently small. Before proving this result, consider the following preliminary lemma.

\begin{lem}\label{WeaklyConvex:Lemma:StrictlyConvex}
Fix $y\in\orbifold{O}$; let $(\widetilde{U},G,\phi)$ be an orbifold chart around $y$ on $\orbifold{O}$ and set $\widetilde{y}=\phi^{-1}(y)$. Suppose that
\begin{enumerate}
\item
	$\kappa\in\real$ is an upper bound for the sectional curvatures on $(\widetilde{U},\widetilde{\metric{g}})$ and
\item
	$\rho_y>0$ is such that $\B^{(\widetilde{U},\widetilde{\metric{g}})}(\widetilde{y},\rho_y)$ is strongly convex on $(\widetilde{U},\widetilde{\metric{g}})$ and $\rho_y<\kappa^{-1/2}\pi/4$ if $\kappa>0$.
\end{enumerate}
We conclude that if $0<\rho<\rho_y$ and $x\in\B^{(\orbifold{O},\metric{g})}(y,\rho)$, then
\begin{equation}\label{WeaklyConvex:Equation:StrictlyConvex}
	(\B^{(\orbifold{O},\metric{g})}(y,\rho),\metric{g})\ni z
	\mapsto
	\mathrm{dist}^{(\orbifold{O},\metric{g})}(x,z)^2\in [0,(2\rho)^2]
\end{equation}
is strictly convex.
\end{lem}
\begin{proof}
First, we want to show that
\begin{multline}\label{ConvenientRadius:Eqn:StrictlyConvex3}
	\text{given}~\widetilde{x}\in\B^{(\widetilde{U},\widetilde{\metric{g}})}(\widetilde{y},\rho_y),
	\\
	(\B^{(\widetilde{U},\widetilde{\metric{g}})}(\widetilde{y},\rho_y),\widetilde{\metric{g}})\ni\widetilde{z}
	\mapsto
	\dist^{(\widetilde{U},\widetilde{\metric{g}})}(\widetilde{x},\widetilde{z})^2\in[0,(2\rho_y)^2[
	~\text{is strictly convex}.
\end{multline}
For that purpose, we will argue as in \cite[Proof of Theorem 1.2]{Karcher1977}. Fix $\widetilde{x}\in\B^{(\widetilde{U},\widetilde{\metric{g}})}(\widetilde{y},\rho_y)$, let $\widetilde{\psi}$ be the diffeomorphism from a star-shaped subset of $\real^n$ to $\B^{(\widetilde{U},\widetilde{\metric{g}})}(\widetilde{y},\rho)$ induced by $\exp_{\widetilde{x}}^{(\widetilde{U},\widetilde{\metric{g}})}$ and let $\widetilde{\gamma}\colon [0,1]\to(\B^{(\widetilde{U},\widetilde{\metric{g}})}(\widetilde{y},\rho_y),\widetilde{\metric{g}})$ be a minimizing geodesic. Given $(s,t)\in ]1-\delta,1+\delta[\times [0,1]$, let
\[
	c(s,t)=\widetilde{\psi}\left(
		s\widetilde{\psi}^{-1}(\widetilde{\gamma}(t))
	\right)\in\B^{(\widetilde{U},\widetilde{\metric{g}})}(\widetilde{y},\rho_y)
\]
and $J_t(s)=\partial_2c(s,t)\in\real^n$, where $\delta>0$ is sufficiently small. Fix $t\in[0,1]$. We have
\begin{align*}
	\left.\dod{}{\tau}
		\dist^{(\widetilde{U},\widetilde{\metric{g}})}(\widetilde{x},\widetilde{\gamma}(\tau))^2
	\right|_{\tau=t}
	&=
	\left.\dod{}{\tau}
		\widetilde{\metric{g}}(\partial_1c(1,\tau),\partial_1c(1,\tau))
	\right|_{\tau=t};\\
	&=
	2\widetilde{\metric{g}}\left(
		\nabla^{(\widetilde{U},\widetilde{\metric{g}})}_{\widetilde{\gamma}'(t)}\partial_1c(1,t),\partial_1c(1,t)
	\right);\\
	&=
	2\int_0^1\widetilde{\metric{g}}\left(
		\nabla^{(\widetilde{U},\widetilde{\metric{g}})}_{J_t(s)}\partial_2c(s,t),\partial_1c(s,t)
	\right)\mathrm{d}s;\\
	&=
	2\widetilde{\metric{g}}(
		\partial_2c(s,t),\partial_1c(s,t)
	)
\end{align*}
and
\[
	\left.\dod[2]{}{\tau}
		\dist^{(\widetilde{U},\widetilde{\metric{g}})}(\widetilde{x},\widetilde{\gamma}(\tau))^2
	\right|_{\tau=t}
	=
	2\widetilde{\metric{g}}(\partial_2c(1,t),\partial_1\partial_2c(1,t))
	=
	2\widetilde{\metric{g}}(
		J_t(1),J_t'(1)
	),
\]
because $\widetilde{\gamma}=\partial_2c(1,\cdot)$ is a geodesic on $(\widetilde{U},\widetilde{\metric{g}})$. The mapping $J_t$ is a Jacobi field, so it follows from well-known estimates concerning Jacobi fields that
\[
	\left.\dod[2]{}{\tau}
		\dist^{(\widetilde{U},\widetilde{\metric{g}})}(\widetilde{x},\widetilde{\gamma}(\tau))^2
	\right|_{\tau=t}
	\geq
	\widetilde{\metric{g}}(
		\widetilde{\gamma}'(t),\widetilde{\gamma}'(t)
	)
\]
if $\kappa\leq 0$ and
\[
	\left.\dod[2]{}{\tau}
		\dist^{(\widetilde{U},\widetilde{\metric{g}})}(\widetilde{x},\widetilde{\gamma}(\tau))^2
	\right|_{\tau=t}
	\geq
	\frac{2\rho_y\sqrt{\kappa}}{\tan(2\rho_y)}
	\widetilde{\metric{g}}(
		\widetilde{\gamma}'(t),\widetilde{\gamma}'(t)
	)
\]
if $\kappa>0$ (see \cite[Appendix A]{Karcher1977}).

Let us prove that if $0<\rho<\rho_y$ and $x\in\B^{(\orbifold{O},\metric{g})}(y,\rho)$, then \eqref{WeaklyConvex:Equation:StrictlyConvex} is strictly convex. It suffices to show that if $\gamma\colon [0,1]\to(\B^{(\orbifold{O},\metric{g})}(y,\rho),\metric{g})$ is a non-constant minimizing geodesic, then
\[
	[0,1]\ni t\mapsto \dist^{(\orbifold{O},\metric{g})}(x,\gamma(t))^2\in [0,(2\rho)^2]
\]
is strictly convex. We have $\B^{(\orbifold{O},\metric{g})}(y,\rho)\subset U$, so we can fix a minimizing geodesic $\widetilde{\gamma}\colon [0,1]\to(\B^{(\widetilde{U},\widetilde{\metric{g}})}(\widetilde{y},\rho),\widetilde{\metric{g}})$ such that $\phi\circ\widetilde{\gamma}=\gamma$ and $\dist^{(\widetilde{U},\widetilde{\metric{g}})}(\widetilde{x},\widetilde{\gamma}(t))=\dist^{(\orbifold{O},\metric{g})}(x,\gamma(t))$ for every $t\in[0,1]$. It follows from \eqref{ConvenientRadius:Eqn:StrictlyConvex3} that
\[
	[0,1]\ni t\mapsto \dist^{(\widetilde{U},\widetilde{\metric{g}})}(\widetilde{x},\widetilde{\gamma}(t))^2\in [0,(2\rho)^2]
\]
is strictly convex, hence the conclusion.
\end{proof}

Now, we can prove that metric balls whose radius is sufficiently small are weakly convex.

\begin{prop}\label{WeaklyConvex:Proposition}
If $y\in\orbifold{O}$, $\rho_y>0$ are such that the conclusion of Lemma \ref{WeaklyConvex:Lemma:StrictlyConvex} holds and $0<\rho<\rho_y$, then $\B^{(\orbifold{O},\metric{g})}(y,\rho/3)$ is weakly convex on $(\orbifold{O},\metric{g})$.
\end{prop}
\begin{proof}
Let us follow the line of argument in \cite[p. 186, Problem 6-5]{Lee2018}. Fix $x,z\in\B^{(\orbifold{O},\metric{g})}(y,\rho/3)$ and let $\gamma\colon [0,1]\to\orbifold{O}$ be a minimizing geodesic linking $x$ to $z$. We have $\dist^{(\orbifold{O},\metric{g})}(x,z)<2\rho/3$, so $\length^{(\orbifold{O},\metric{g})}(\gamma)<2\rho/3$. As $\dist^{(\orbifold{O},\metric{g})}(x,y)<\rho/3$, we obtain $\gamma([0,1])\subset\B^{(\orbifold{O},\metric{g})}(x,\rho)$. By hypothesis,
\[
	[0,1]\ni t
	\mapsto
	\mathrm{dist}^{(\orbifold{O},\metric{g})}(y,\gamma(t))^2\in [0,\rho^2]
\]
is strictly convex, so its maximum point is in $\{0,1\}$ and thus
\[
	\max_{0\leq t\leq 1}\mathrm{dist}^{(\orbifold{O},\metric{g})}(y,\gamma(t))<\rho/3,\]
which implies $\gamma([0,1])\subset\B^{(\orbifold{O},\metric{g})}(x,\rho/3)$.
\end{proof}

\subsection*{The Riemannian center of mass}
Suppose that $(M,\metric{g})$ is a complete Riemannian manifold with $\inj(M,\metric{g}):=\inf\inj^{(M,\metric{g})}>0$ and $\kappa\in\real$ is an upper bound for the sectional curvatures of $(M,\metric{g})$. It then follows from \cite[Theorem 2.1]{Afsari2011} that if $y\in M$, $\rho\in]0,\inj(M,\metric{g})/2[$ is such that $\rho<\kappa^{-1/2}\pi/2$ if $\kappa>0$ and $\nu$ is a probability measure on $M$ with $\supp\nu\subset\B^{(M,\metric{g})}(y,\rho)$, then
\[
	M\ni x\mapsto\int_{\B^{(M,\metric{g})}(y,\rho)} \dist^{(M,\metric{g})}(x,z)^2\mathrm{d}\nu(z)
\]
admits a unique minimum point $x_\nu\in\B^{(M,\metric{g})}(y,\rho)$, which we call the \emph{Riemannian center of mass} of $\nu$.

Our goal is to prove that under hypotheses similar to those in the previous paragraph, we can define the Riemannian center of mass for probability measures on the complete Riemannian orbifold $(\orbifold{O},\metric{g})$. Suppose that we can fix $\kappa\in\real$ such that
\begin{enumerate}[label=($\thesection.\arabic*$)]
\setcounter{enumi}{3}
\item \label{H1}
	if $\phi\colon\widetilde{U}\to\orbifold{O}$ is a chart map, then $\kappa$ is an upper bound for the sectional curvatures on $(\widetilde{U},\widetilde{\metric{g}})$
\end{enumerate}
and we can fix $\rho>0$ such that
\begin{enumerate}[label=($\thesection.\arabic*$)]
\setcounter{enumi}{4}
\item \label{H2}
	$3\rho<\kappa^{-1/2}\pi/4$ if $\kappa>0$ and
\item \label{H3}
	given $y\in\orbifold{O}$, there exists $\phi\colon\widetilde{U}\to U$, a chart map around $y$ on $\orbifold{O}$ such that $\B^{(\widetilde{U},\widetilde{\metric{g}})}(\widetilde{y},3\rho)$ is strongly convex on $(\widetilde{U},\widetilde{\metric{g}})$, where $\widetilde{y}:=\phi^{-1}(y)$.
\end{enumerate}
\setcounter{equation}{6}
In this context, the following proposition is the main result of this section.

\begin{prop}\label{CenterOfMass:Proposition}
If $y\in\orbifold{O}$ and $\nu$ is a probability measure on $\orbifold{O}$ supported on $\B^{(\orbifold{O},\metric{g})}(y,\rho)$, then
\begin{equation}\label{CenterOfMass:Eqn}
	\orbifold{O}\ni x\mapsto \int_{\B^{(\orbifold{O},\metric{g})}(y,\rho)} \dist^{(\orbifold{O},\metric{g})}(x,z)^2\mathrm{d}\nu(z),
\end{equation}
admits a unique minimum point $x_\nu\in\B^{(\orbifold{O},\metric{g})}(y,\rho)$, which we call the \emph{Riemannian center of mass} of $\nu$.
\end{prop}

First, we need to prove that a minimum point of \eqref{CenterOfMass:Eqn} is necessarily in $\B^{(\orbifold{O},\metric{g})}(y,\rho)$.
\begin{lem}\label{CenterOfMass:Lemma:InteriorMinimum}
If $y\in\orbifold{O}$, $\nu$ is a probability measure on $\orbifold{O}$ supported on $\B^{(\orbifold{O},\metric{g})}(y,\rho)$ and $x\in\orbifold{O}$ is a minimum point of \eqref{CenterOfMass:Eqn}, then $x\in\B^{(\orbifold{O},\metric{g})}(y,\rho)$.
\end{lem}
\begin{proof}
This proof is similar to \cite[Proof of Theorem 2.1]{Afsari2011}.

We claim that $x\in\overline{\B^{(\orbifold{O},\metric{g})}(y,\rho)}$. In fact, it suffices to prove that if $x\in\orbifold{O}\setminus\overline{\B^{(\orbifold{O},\metric{g})}(y,\rho)}$, then we can fix $x'\in\B^{(\orbifold{O},\metric{g})}(y,\rho)$ which guarantees that $\dist^{(\orbifold{O},\metric{g})}(x',z)<\dist^{(\orbifold{O},\metric{g})}(x,z)$ for every $z\in\B^{(\orbifold{O},\metric{g})}(y,\rho)$. If $\dist^{(\orbifold{O},\metric{g})}(x,y)\geq 2\rho$, then it follows from the triangle inequality that it suffices to take $x'=y$. Now, consider the case $\rho<\dist^{(\orbifold{O},\metric{g})}(x,y)<2\rho$. Given $z\in\B^{(\orbifold{O},\metric{g})}(y,\rho)$, it holds that
\[
	\dist^{(\orbifold{O},\metric{g})}(x,z)=\min\{
		\dist^{(\widetilde{U},\widetilde{\metric{g}})}(\widetilde{x},\widetilde{z}):\widetilde{x}\in\phi^{-1}(x),\widetilde{z}\in\phi^{-1}(z)
	\},
\]
so we only have to show that given $\widetilde{x}\in\phi^{-1}(x)$, there exists $\widetilde{x}'\in\B^{(\widetilde{U},\widetilde{\metric{g}})}(\widetilde{y},\rho)$ such that $\dist^{(\widetilde{U},\widetilde{\metric{g}})}(\widetilde{x}',\widetilde{z})<\dist^{(\widetilde{U},\widetilde{\metric{g}})}(\widetilde{x},\widetilde{z})$ for any $\widetilde{z}\in\B^{(\widetilde{U},\widetilde{\metric{g}})}(\widetilde{y},\rho)$, which can be proved by arguing as on \cite[p. 662]{Afsari2011}.

We already proved that $x\in\overline{\B^{(\orbifold{O},\metric{g})}(y,\rho)}$, so it only remains to show that $x\not\in\partial(\B^{(\orbifold{O},\metric{g})}(y,\rho))$. Given $z,w\in\orbifold{O}$, let $\gamma_{z,w}\colon [0,1]\to\orbifold{O}$ be a minimal geodesic on $(\orbifold{O},\metric{g})$ linking $z$ to $w$. Let us argue similarly as on \cite[p. 661]{Afsari2011} to prove that if $x\in\partial(\B^{(\orbifold{O},\metric{g})}(y,\rho))$, then
\[
	[T,1]\ni t
	\mapsto	
	\int_{\B^{(\orbifold{O},\metric{g})}(y,\rho)} \dist^{(\orbifold{O},\metric{g})}(\gamma_{y,x}(t),w)^2|u(w)|\mathrm{d}\mu^{(\orbifold{O},\metric{g})}(w)
\]
is strictly increasing for a certain $T\in]0,1[$. Indeed, fix $T\in]0,1[$ for which
\[
	\supp u
	\subset
	\B^{(\orbifold{O},\metric{g})}\left(y,\dist^{(\orbifold{O},\metric{g})}(\gamma_{y,x}(T),y)\right)
	\subset
	\B^{(\orbifold{O},\metric{g})}(y,\rho).
\]
If $T\leq s<1$ and $z\in\B^{(\orbifold{O},\metric{g})}\left(y,\dist^{(\orbifold{O},\metric{g})}(\gamma_{y,x}(T),y)\right)$, then
\[
	\gamma_{z,\gamma_{y,x}(s)}(]0,1[)\subset\B^{(\orbifold{O},\metric{g})}\left(y,\dist^{(\orbifold{O},\metric{g})}(\gamma_{y,x}(T),y)\right)
\]
due to Proposition \ref{WeaklyConvex:Proposition} and thus $\metric{g}(\gamma_{z,\gamma_{y,x}(s)}'(1^-),\gamma_{y,x}'(s))>0$, so
\[
	[T,1]\ni t\mapsto\dist^{(\orbifold{O},\metric{g})}(\gamma_{y,x}(t),z)
\]
is strictly increasing.
\end{proof}

Now, we just need to prove the existence and uniqueness of the minimum point to conclude.

\begin{proof}[Proof of Proposition \ref{CenterOfMass:Proposition}]
The function \eqref{CenterOfMass:Eqn} is continuous and $\orbifold{O}$ is compact, so it admits a minimum point. Due to Lemma \ref{CenterOfMass:Lemma:InteriorMinimum}, the minima of \eqref{CenterOfMass:Eqn} are in $\B^{(\orbifold{O},\metric{g})}(y,\rho)$, so it suffices to prove that
\begin{equation}\label{CenterOfMass:Equation:Auxiliary}
	(\B^{(\orbifold{O},\metric{g})}(y,\rho),\metric{g})\ni x\mapsto \int_{\B^{(M,\metric{g})}(y,\rho)} \dist^{(M,\metric{g})}(x,z)^2\mathrm{d}\nu(z)
\end{equation}
has a unique minimum point (note that the domain of \eqref{CenterOfMass:Equation:Auxiliary} is different from the domain of \eqref{CenterOfMass:Eqn}). It is a corollary of Lemma \ref{WeaklyConvex:Lemma:StrictlyConvex} that given $x\in\B^{(\orbifold{O},\metric{g})}(y,\rho)$,
\[
	(\B^{(\orbifold{O},\metric{g})}(y,\rho),\metric{g})\ni z
	\mapsto
	\mathrm{dist}^{(\orbifold{O},\metric{g})}(x,z)^2\in [0,(2\rho)^2[
\]
is strictly convex, so it follows from the discussion in \cite[Section 3.2.1]{Boyd2004} that \eqref{CenterOfMass:Equation:Auxiliary} is also is strictly convex. The set $\B^{(\orbifold{O},\metric{g})}(x,\rho)$ is weakly convex on $(\orbifold{O},\metric{g})$ due to Proposition \ref{WeaklyConvex:Proposition}, so the result follows from Proposition \ref{LengthSpaces:Proposition:StrictlyConvex}.
\end{proof}

\section{Preliminaries}\label{Preliminaries}

In this section, we develop the technical preliminaries to prove Theorem \ref{Introduction:Theorem:ToProve} in the context presented at the introduction.

\subsection*{The \emph{convenient radius} $\rho$}
If we have $\orbifold{O}^{\sing}\neq\emptyset$, then Proposition \ref{ExponentialMap:Proposition:DiscontinuousInj} implies $\inf\inj^{(\orbifold{O},\metric{g})}=0$, so we cannot fix $\rho\in]0,\inf\inj^{(\orbifold{O},\metric{g})}[$ as in \cite[Remark 2.7]{Benci2007}. Our goal then becomes fixing an alternative \emph{convenient radius} $\rho>0$ that will ensure that metric balls on $(\orbifold{O},\metric{g})$ with radius $\rho$ have convenient geometric properties in a sense that will soon become clear.

Consider the smooth fiber bundle over $\orbifold{O}^{\sing}$ with fiber
\[
	N_x:=\{v\in \T_x\orbifold{O}: v\perp^{(\orbifold{O},\metric{g})}\T_x\orbifold{O}^{\sing}\}.
\]
Let us use the fiber bundle $N$ to construct a topological tubular neighborhood of $\orbifold{O}^{\sing}$ on $\orbifold{O}$.

\begin{lem}\label{ConvenientRadius:Lemma:Homeomorphism}
There exists $\rho>0$ such that
\[
	\left\{
		w\in N: \metric{g}(w,w)<(3\rho)^2
	\right\}
	\ni
	v
	\mapsto
	\exp^{(\orbifold{O},\metric{g})}(v)
	\in
	\B^{(\orbifold{O},\metric{g})}(\orbifold{O}^{\sing},3\rho)
\]
is a homeomorphism.
\end{lem}
\begin{proof}
We know that $\orbifold{O}^{\sing}$ is a compact subset of $\orbifold{O}$ and the restriction of $\inj^{(\orbifold{O},\metric{g})}$ to $\orbifold{O}^{\sing}$ is continuous (see Remark \ref{ExponentialMap:Remark}), so it suffices to prove that if $\Sigma$ is a singular stratum of $\orbifold{O}$ and $x\in\Sigma$, then there exists $\rho_x>0$ such that
\[
	\left\{
		w\in \left.N\right|_{\Sigma\cap\B^{(\orbifold{O},\metric{g})}(x,\rho_x)}: \metric{g}(w,w)<\rho_x^2
	\right\}
	\ni
	v
	\mapsto
	\exp^{(\orbifold{O},\metric{g})}(v)
	\in
	\orbifold{O}
\]
is a homeomorphism with its image. Let $k=\dim\Sigma$; let $(\real^n,G,\phi)$ be a linear chart adapted to $\Sigma$ and centered at $x$ on $\orbifold{O}$ and let $E=\phi^*N$, i.e., $E$ is the vector sub-bundle of
\[
	\left.\T\real^n\right|_{\real^k\times\{0_{\real^{n-k}}\}}
	=
	(\real^k\times\{0_{\real^{n-k}}\})\times\real^n
\]
with rank $(n-k)$ and fiber
\[
	E_{\widetilde{y}}:=\{
		v\in\real^n: v\perp^{(\real^n,\widetilde{\metric{g}}_{\widetilde{y}})}\real^k\times\{0_{\real^{n-k}}\}
	\}
\]
for any $\widetilde{y}\in\real^k\times\{0_{\real^{n-k}}\}=\phi^{-1}(\Sigma)$. We can argue as in the proof of (2) in \cite[Proposition 5.5.1]{Petersen2016} to prove that the mapping
\[
	E\ni (\widetilde{y},v)\mapsto\psi(\widetilde{y},v):=\exp^{(\real^n,\widetilde{\metric{g}})}(\widetilde{y},v)\in\real^n.
\]
is such that $\mathrm{d}\psi_{(0,0)}$ is invertible. At this point, it suffices to argue as in the proof of \cite[Proposition 6.8]{Borzellino2008} to conclude.
\end{proof}

Due to compactness, we can fix $\kappa\in\real$ such that \ref{H1} holds and we can define the notion of a convenient radius already.

\begin{defn}\label{ConvenientRadius:Definition:Second}
Fix a \emph{convenient radius} $\rho\in]0,\inf_{\Z^{\orbifold{O}}}\inj^{(\orbifold{O},\metric{g})}[$ such that \ref{H2}, \ref{H3} and the conclusion of Lemma \ref{ConvenientRadius:Lemma:Homeomorphism} hold.
\end{defn}

In fact, if $x\in\orbifold{O}$ is $3\rho$-close to $\orbifold{O}^{\sing}$, then it admits a unique closest point on $\orbifold{O}^{\sing}$. As the singular strata of $\orbifold{O}$ are totally geodesic submanifolds of $(\orbifold{O},\metric{g})$, it turns out that the proof of this result is very similar to the proof of \cite[Theorem 5.6.21]{Petersen2016}.

\begin{lem}\label{ConvenientRadius:Lemma:ClosestPoint}
If $\dist^{(\orbifold{O},\metric{g})}(x,\orbifold{O}^{\sing})<3\rho$, then there exists a unique $x^{\sing}\in\orbifold{O}^{\sing}$ such that $\dist^{(\orbifold{O},\metric{g})}(x,\orbifold{O}^{\sing})=\dist^{(\orbifold{O},\metric{g})}(x,x^{\sing})$.
\end{lem}
\begin{proof}
The fact that there exists a closest point $x^{\sing}\in\orbifold{O}^{\sing}$ follows from compactness. Let $\Sigma$ be the singular stratum of $\orbifold{O}$ which contains $x^{\sing}$. The lemma is obvious when $x\in\Sigma$, so suppose otherwise.

Let us preliminarly prove that
\begin{equation}\label{ConvenientRadius:Preliminary}
[\exp^{(\orbifold{O},\metric{g})}_{x^{\sing}}]^{-1}(x)\in N_{x^{\sing}}.
\end{equation}
Let $\gamma\colon ]-\delta,\delta[\to\Sigma$ be a smooth curve such that $\gamma(0)=x^{\sing}$ and define
\[
	H(s,t)=\exp^{(\orbifold{O},\metric{g})}_{\gamma(s)}\left[t(\exp^{(\orbifold{O},\metric{g})}_{\gamma(s)})^{-1}(x)\right],
\]
for any $(s,t)\in]-\delta,\delta[\times[0,1]$. Given $s\in ]-\delta,\delta[$, we define the energy of $H(s,\cdot)$ as
\[
	E(s)=\int_0^1 \metric{g}(\partial_2H(s,t),\partial_2H(s,t))\mathrm{d}t.
\]

We claim that $H(s,t)\in\orbifold{O}^{\reg}$ whenever $t\neq 0$. Indeed, suppose by contradiction that there exists $s\in]-\delta,\delta[$ such that $\partial_2H(s,0)\in \T_x\Sigma$. It follows that $H(s,t)\in\Sigma$ for every $t\in [0,1]$ because $\Sigma$ is totally geodesic on $(\orbifold{O},\metric{g})$ (see \cite[Proposition 32]{Borzellino1992}), which contradicts the fact that $H(s,1)=x\in\orbifold{O}^{\reg}$.

It follows from the claim in the previous paragraph that the formula for the first variation of energy holds. As $\dist^{(\orbifold{O},\metric{g})}(x,x^{\sing})=\dist^{(\orbifold{O},\metric{g})}(x,\orbifold{O}^{\sing})$, we obtain
\[
	\metric{g}\left(
		\gamma'(0),
		[\exp^{(\orbifold{O},\metric{g})}_{x^{\sing}}]^{-1}(x)
	\right)=E'(0)=0
\]
and thus \eqref{ConvenientRadius:Preliminary} holds.

Let us prove the uniqueness of the closest point. Suppose that $x_1^{\sing},x_2^{\sing}\in\Sigma$ are such that
\[
	\dist^{(\orbifold{O},\metric{g})}(x,\orbifold{O}^{\sing})=\dist^{(\orbifold{O},\metric{g})}(x,x_1^{\sing})=\dist^{(\orbifold{O},\metric{g})}(x,x_2^{\sing}).
\]
Assertion \eqref{ConvenientRadius:Preliminary} implies $[\exp_{x_1^{\sing}}^{(\orbifold{O},\metric{g})}]^{-1}(x),[\exp_{x_2^{\sing}}^{(\orbifold{O},\metric{g})}]^{-1}(x) \in N$, so Lemma \ref{ConvenientRadius:Lemma:Homeomorphism} shows that $[\exp_{x_1^{\sing}}^{(\orbifold{O},\metric{g})}]^{-1}(x)=[\exp_{x_2^{\sing}}^{(\orbifold{O},\metric{g})}]^{-1}(x)$ and thus $x_1^{\sing}=x_2^{\sing}$.
\end{proof}

In particular, it follows that the projection of points which are $3\rho$-close to $\orbifold{O}^{\sing}$ to the closest singular point,
\[
	\B^{(\orbifold{O},\metric{g})}(\orbifold{O}^{\sing},3\rho)\ni x\mapsto x^{\sing}\in\orbifold{O}^{\sing},
\]
is a well-defined continuous map.

\subsection*{The limiting problem in $\real^n$}

Consider the following nonlinear problem in $\real^n$,
\begin{equation}\label{FunctionalFramework:Eqn:Rn}
	\begin{cases}
		-\Delta u+u=u|u|^{p-2}~\text{and}\\
		u>0.
	\end{cases}
\end{equation}
A weak solution in $H^1(\real^n)$ to \eqref{FunctionalFramework:Eqn:Rn} is precisely a critical point of $E\in C^2(H^1(\real^n))$ given by
\[
	E(u):=\int_{\real^n} \frac{|\grad^{\real^n}(u)(\widetilde{x})|^2}{2}+\frac{u(\widetilde{x})^2}{2}-\frac{u^+(\widetilde{x})^p}{p}\mathrm{d}\widetilde{x}.
\]
We define $m(E)=\inf_{u\in\mathcal{N}(E)} E(u)$, where
\[
	\mathcal{N}(E):=\{u\in H^1(\real^n)\setminus\{0\}:\mathrm{d}_uE(u)=0\}
\]is the Nehari manifold associated to $E$. It is actually well known that there exists a unique element of
\[
	H_{\rad}^1(\real^n):=\{
		v\in H^1(\real^n):~\text{for a.e.}~x,y\in\real^n,~|x|=|y|~\text{implies}~v(x)=v(y)
	\}
\]
which is a weak solution to \eqref{FunctionalFramework:Eqn:Rn} and minimizes $E$.

\begin{defn}
We will denote by $\widetilde{V}$ the unique function in $H_{\rad}^1(\real^n)\cap\mathcal{N}(E)$ such that $E(\widetilde{V})=m(E)$.
\end{defn}

If $0<\epsilon<1$, then we define $\widetilde{V}_{\epsilon}(\widetilde{x})=\widetilde{V}(\epsilon^{-1}\widetilde{x})$ for a.e. $\widetilde{x}\in\real^n$, so that $\widetilde{V}_\epsilon\in H^1_{\rad}(\real^n)$ is a weak solution to
\[
	\begin{cases}
		-\epsilon^2\Delta u+u=u|u|^{p-2}~\text{and}\\
		u>0.
	\end{cases}
\]

\subsection*{The injection $i_\epsilon\colon \Z^\orbifold{O}\to\mathcal{N}_\epsilon$}

Let us begin by fixing the notation for cut-off functions.
\begin{defn}\label{TheInjection:Defn:CutOff}
Given $r>0$, $\chi_r\colon [0,\infty[ \to [0,1]$ denotes a nonincreasing smooth function such that $\chi_r(t)=1$ if $t\leq r/2$; $\chi_r(t)=0$ if $t\geq r$ and $\sup |\chi_r'|\leq 3/r$.
\end{defn}

We can exploit the radial symmetry of $\widetilde{V}$ to obtain the following definition.

\begin{defn}\label{TheInjection:Definition:Vepsilonx}
Suppose that $0<\epsilon<1$, $x\in \Z^{\orbifold{O}}$ and
$(\widetilde{U}:=\B^{\real^n}(0,\rho),G,\phi)$ is a normal chart centered at $x$ on $(\orbifold{O},\metric{g})$. As $\widetilde{V}_\epsilon\in H^1_{\rad}(\real^n)$ and $G\subset\mathrm{O}(\real^n)$, we can define
\[
	V_{\epsilon,x}(y)
	=
	\begin{cases}
		\widetilde{V}_{\epsilon}(\phi^{-1}(y))\chi_\rho(\dist^{(\orbifold{O},\metric{g})}(x,y))~&\text{if}~y\in U~\text{and}\\
		0~&\text{otherwise}
	\end{cases}
\]
for $\mu^{(\orbifold{O},\metric{g})}$-a.e. $y\in\orbifold{O}$.
\end{defn}

Let us prove that $V_{\epsilon,x}\in H^1(\orbifold{O},\metric{g})$.

\begin{lem}
Given $(\epsilon,x)\in ]0,1[\times\Z^{\orbifold{O}}$, it holds that $V_{\epsilon,x}\in H^1(\orbifold{O},\metric{g})$.
\end{lem}
\begin{proof}
It suffices to construct $\{V_k\}_{k\in\nat}\subset H^1(\orbifold{O},\metric{g})$ such that $(V_k)_{k\in\nat}$ is a Cauchy sequence in $H^1(\orbifold{O},\metric{g})$ and $V_k\to V_{\epsilon,x}$ $\mu^{(\orbifold{O},\metric{g})}$-a.e. as $k\to\infty$.

Let $(\widetilde{U}:=\B^{\real^n}(0,\rho),G,\phi)$ be a normal chart centered at $x$ on $(\orbifold{O},\metric{g})$. Fix $\{\widetilde{V_k}\}_{k\in\nat}\subset C_{\rad}^\infty(\widetilde{U})$ such that $\supp\widetilde{V_k}\subset\widetilde{U}$ for any $k\in\nat$ and $\widetilde{V_k}\to \widetilde{V}_\epsilon$ in $H^1(\widetilde{U})$ as $k\to\infty$. Given $k\in\nat$, let $V_k\in C^\infty(\orbifold{O})$ be obtained from $\widetilde{V_k}\colon \widetilde{U}\to\real$ analogously as in Definition \ref{TheInjection:Definition:Vepsilonx}.

We claim that $(V_k)_{k\in\nat}$ is a Cauchy sequence in $H^1(\orbifold{O},\metric{g})$. Inded, the inclusion $\overline{\B^{(\T_x\orbifold{O},\metric{g}_x)}(0,\rho)}\subset\Omega^{(\orbifold{O},\metric{g})}$ shows that there exists $K>0$ such that
\begin{multline*}
	\|V_{k_1}-V_{k_2}\|_{H^1(\orbifold{O},\metric{g})}^2
	=
	\\
	=
	\frac{1}{|G|}\int_{\widetilde{U}}
		\left[
			|\grad^{\real^n}(\widetilde{V_{k_1}}-\widetilde{V_{k_2}})|^{(\widetilde{U},\widetilde{\metric{g}})}
		\right]^2
		+
		(\widetilde{V_{k_1}}-\widetilde{V_{k_2}})^2
	\mathrm{d}\mu^{(\widetilde{U},\widetilde{\metric{g}})}
	\leq
	\\
	\leq
	\frac{K}{|G|}\|\widetilde{V_{k_1}}-\widetilde{V_{k_2}}\|_{H^1(\widetilde{U})}^2
\end{multline*}
for every $k_1,k_2\in\nat$.

Let us prove that, up to subsequence, $V_k\to V_{\epsilon,x}$ $\mu^{(\orbifold{O},\metric{g})}$-a.e. as $k\to\infty$. Since $\widetilde{V_k}\to \widetilde{V}_\epsilon$ in $H^1(\real^n)$ as $k\to\infty$, we conclude that, up to subsequence, $\widetilde{V_k}\to \widetilde{V}_\epsilon$ a.e. as $k\to\infty$. The result is therefore a corollary of the fact that  given $k\in\nat$, $\widetilde{V_k}$ is a local lift of $V_k$ and analogously for the pair $\widetilde{V}_{\epsilon}$, $V_{\epsilon,x}$.
\end{proof}

The following remark is usual whenever the Nehari manifold is being considered.

\begin{rmk}\label{TheInjection:Rmk:n_epsilon}
Suppose that $0<\epsilon<1$ and $u\in H^1(\orbifold{O},\metric{g})$ is such that $u^+$ is not zero $\mu^{(\orbifold{O},\metric{g})}$-a.e. In this situation, $su\in\mathcal{N}_\epsilon$ if, and only if, $s=n_\epsilon(u)$, where
\begin{multline*}
	n_\epsilon(u)^{p-2}:=
	\\
	:=
	\left[
		\int_{\orbifold{O}}(u^+)^{p}\mathrm{d}\mu^{(\orbifold{O},\metric{g})}
	\right]^{-1}
	\int_{\orbifold{O}}
		\epsilon^2 \metric{g}(\grad^{(\orbifold{O},\metric{g})}(u),\grad^{(\orbifold{O},\metric{g})}(u))+u^2
	\mathrm{d}\mu^{(\orbifold{O},\metric{g})}>0.
\end{multline*}
\end{rmk}

We can finally define the injection $i_\epsilon\colon\Z^{\orbifold{O}}\to\mathcal{N}_\epsilon$.

\begin{defn}
If $0<\epsilon<1$, then we use the previous remark to define
\[
	\Z^{\orbifold{O}}\ni
	x\mapsto i_\epsilon(x):=n_\epsilon(V_{\epsilon,x})V_{\epsilon,x}
	\in\mathcal{N}_\epsilon.
\]
\end{defn}

It suffices to argue as in the proof of the first conclusion of \cite[Proposition 4.2]{Benci2007} to obtain the following result.

\begin{lem}
Given $\epsilon\in]0,1[$, $i_\epsilon\colon \Z^{\orbifold{O}}\to\mathcal{N}_\epsilon$ is continuous.
\end{lem}

The next result follows from a little adjustment to the proof of the second conclusion of \cite[Proposition 4.2]{Benci2007}.

\begin{lem}\label{TheInjection:Lem:Energy}
Given $\delta>0$, there exists $\epsilon_\delta\in]0,1]$ such that if $0<\epsilon<\epsilon_\delta$, then $i_\epsilon(x)\in \Lambda_{\epsilon,(\zeta^{\orbifold{O}})^{-1}m(E)+\delta}$ for any $x\in\Z^{\orbifold{O}}$.
\end{lem}
\begin{proof}
It suffices to argue as in the proof of \cite[Proposition 4.2]{Benci2007} to prove that given $\delta>0$, there exists $\epsilon_\delta\in]0,1]$ such that if $0<\epsilon<\epsilon_\delta$, then
\[
	\left|
		\epsilon^{-n}\|V_{\epsilon,x}\|_{L^2(\orbifold{O},\metric{g})}^2
		-
		(\zeta^\orbifold{O})^{-1}\|\widetilde{V}\|^2_{L^2(\real^n)}
	\right|<\delta;
\]
\[
	\left|
		\epsilon^{2-n}\|\grad^{(\orbifold{O},\metric{g})}(V_{\epsilon,x})\|_{L^2(\orbifold{O},\metric{g})}^2
		-
		(\zeta^\orbifold{O})^{-1}\|\grad^{\real^n}(\widetilde{V})\|^2_{L^2(\real^n)}
	\right|<\delta
\]
and
\[
	\left|
		\epsilon^{-n}\|V_{\epsilon,x}\|_{L^p(\orbifold{O},\metric{g})}^p
		-
		(\zeta^\orbifold{O})^{-1}\|\widetilde{V}\|^p_{L^p(\real^n)}
	\right|<\delta
\]
(see \cite[(4.5)--(4.7)]{Benci2007}).

In view of Remark \ref{TheInjection:Rmk:n_epsilon}, it is a corollary of the previous paragraph that given $\delta>0$, there exists $\epsilon_\delta\in]0,1]$ such that given $(\epsilon,x)\in]0,\epsilon_\delta[\times \Z^{\orbifold{O}}$, we have
\[
	|n_{\epsilon}(V_{\epsilon,x})-1|<\delta.
\]
As
\[
	J_\epsilon(n_\epsilon(V_{\epsilon,x})V_{\epsilon,x})
	=
	\frac{1}{\epsilon^n}\frac{p-2}{2p}n_\epsilon(V_{\epsilon,x})^p\|V_{\epsilon,x}^+\|_{L^p(\orbifold{O},\metric{g})}^p
\]
for any $(\epsilon,x)\in ]0,1[\times \Z^{\orbifold{O}}$, we conclude that given $\delta>0$, there exists $\epsilon_\delta\in]0,1]$ such that
\[
	\left|
		J_\epsilon(n_\epsilon(V_{\epsilon,x})V_{\epsilon,x})-(\zeta^{\orbifold{O}})^{-1}m(E)
	\right|<\delta.
\]
for any $(\epsilon,x)\in]0,\epsilon_\delta[\times \Z^{\orbifold{O}}$.
\end{proof}

\subsection*{Concentration of functions in $\Lambda_{\epsilon,(\zeta^{\orbifold{O}})^{-1}m(E)+\delta}$}
We want to prove the following result.

\begin{thm}\label{Concentration:Thm}
Given $\eta\in]0,1[$, there exist $\delta_\eta\in]0,\infty[$ and $\epsilon_\eta\in]0,1]$ such that if $0<\delta<\delta_\eta$, $0<\epsilon<\epsilon_\eta$ and $u\in \Lambda_{\epsilon,(\zeta^{\orbifold{O}})^{-1}m(E)+\delta}$, then
\[
	\int_{B^{(\orbifold{O},\metric{g})}(x,\rho)} (u^+)^p\mathrm{d}\mu^{(\orbifold{O},\metric{g})}>\eta\|u^+\|_{L^p(\orbifold{O},\metric{g})}^p
\]
for a certain $x\in\B^{(\orbifold{O},\metric{g})}(\Z^{\orbifold{O}},\rho)$.
\end{thm}

Note that unlike the analogous results \cite[Proposition 5.5]{Benci2007} and \cite[Theorem 3.7]{Petean2019}, our theorem shows that any $u\in\Lambda_{\epsilon,(\zeta^{\orbifold{O}})^{-1}m(E)+\delta}$ is concentrated around a point situated at a specific subset of $\orbifold{O}$, namely, $\B^{(\orbifold{O},\metric{g})}(\Z^{\orbifold{O}},\rho)$. For its proof, we will follow \cite[Section 5]{Benci2007} with minor modifications. The following preliminary lemmas are similar to \cite[Lemmas 5.3,5.4]{Benci2007} and may be proved accordingly.

\begin{lem}\label{Concentration:Lem:FixedPart}
There exists $\gamma>0$ for which given $\delta\in]0,\infty[$, there exists $\epsilon_\delta\in]0,1]$ such that if $0<\epsilon<\epsilon_\delta$ and $u\in\Lambda_{\epsilon,(\zeta^{\orbifold{O}})^{-1}m(E)+\delta}$, then
\[
	\frac{1}{\epsilon^n}\int_{\B^{(\orbifold{O},\metric{g})}(x,	\epsilon)} (u^+)^p\mathrm{d}\mu^{(\orbifold{O},\metric{g})}\geq\gamma
\]
for a certain $x\in\orbifold{O}$.
\end{lem}

\begin{lem}\label{Concentration:Lem:Ekeland}
Fix $\epsilon\in]0,1[$. Given $\delta>0$ and $u\in\mathcal{N}_\epsilon$ such that
\[
	J_\epsilon(u)<\min\left(
		m(J_\epsilon)+2\delta,\frac{m(E)}{\zeta^{\orbifold{O}}}+\delta
	\right),
\]
there exists $u_\delta\in\mathcal{N}_\epsilon$ such that $J_\epsilon(u_\delta)<J_\epsilon(u)$;
\[
	\frac{1}{\epsilon^n}
	\int_{\orbifold{O}}
		\epsilon^2[|\grad^{(\orbifold{O},\metric{g})}(u_\delta-u)|^{(\orbifold{O},\metric{g})}]^2
		+
		(u_\delta-u)^2
	\mathrm{d}\mu^{(\orbifold{O},\metric{g})}
	<
	16\delta
\]
and
\[
	\mathrm{d}_{u_\delta}J_\epsilon(w)^2
	<
	\frac{\delta}{\epsilon^n}
	\int_{\orbifold{O}}
		\epsilon^2\metric{g}(\grad^{(\orbifold{O},\metric{g})}(w),\grad^{(\orbifold{O},\metric{g})}(w))
		+
		w^2
	\mathrm{d}\mu^{(\orbifold{O},\metric{g})}.
\]
for any $w\in\T_{u}\mathcal{N}_\epsilon\subset H^1(\orbifold{O},\metric{g})$.
\end{lem}

Theorem \ref{Concentration:Thm} will be obtained as a corollary of the following result.

\begin{lem}\label{Concentration:Lemma:2}
Given $\eta\in]0,1[$, there exist $\delta_\eta>0$ and $\epsilon_\eta\in]0,1]$ such that if $0<\delta<\delta_\eta$, $0<\epsilon<\epsilon_\eta$ and $u\in \Lambda_{\epsilon,(\zeta^{\orbifold{O}})^{-1}m(E)+\delta}$, then
\[
	\frac{1}{\epsilon^n}\int_{B^{(\orbifold{O},\metric{g})}(x,\rho)} (u^+)^p\mathrm{d}\mu^{(\orbifold{O},\metric{g})}
	>
	\eta\frac{2p}{p-2}\frac{m(E)}{\zeta^{\orbifold{O}}}
\]
for a certain $x\in\B^{(\orbifold{O},\metric{g})}(\Z^{\orbifold{O}},\rho)$.
\end{lem}
\begin{proof}
The result is the same as \cite[Proposition 5.5]{Benci2007} if $\orbifold{O}$ is a manifold, so let us suppose that $\orbifold{O}$ is not a manifold, i.e., $\zeta^\orbifold{O}>1$.

Let us argue by contradiction. Suppose that there exist $\eta\in]0,1[$ and sequences $\{\epsilon_k\}_{k\in\nat},\{\delta_k\}_{k\in\nat}\subset]0,1[,\{u_k\in\mathcal{N}_{\epsilon_k}\}_{k\in\nat}$ such that $\epsilon_k,\delta_k\to 0$ as $k\to\infty$ and given $k\in\nat$,
\[
	J_{\epsilon_k}(u_k)<\min\left(
		m(J_{\epsilon_k})+2\delta_k,\frac{m(E)}{\zeta^{\orbifold{O}}}+\delta_k
	\right)
\]
and
\[
	\frac{1}{\epsilon_k^n}\int_{B^{(\orbifold{O},\metric{g})}(y,\rho)} (u_k^+)^p\mathrm{d}\mu^{(\orbifold{O},\metric{g})}
	\leq
	\eta\frac{2p}{p-2}\frac{m(E)}{\zeta^{\orbifold{O}}}+\delta_k
\]
for any $y\in\B^{(\orbifold{O},\metric{g})}(Z^{\orbifold{O}},\rho)$. Given $k\in\nat$, we can use Lemma \ref{Concentration:Lem:FixedPart} to fix $x_k\in\orbifold{O}$ such that
\[
	\frac{1}{\epsilon_k^n}\int_{\B^{(\orbifold{O},\metric{g})}(x_k,\epsilon_k)} (u^+)^p\mathrm{d}\mu^{\metric{g}}>\gamma.
\]

As $\orbifold{O}$ is compact, there exists $x_\infty\in\orbifold{O}$ such that, up to subsequence, $x_k\to x_\infty$ as $k\to\infty$. Up to discarding a finite number of indices, we can suppose that $\{x_k\}_{k\in\nat}\subset\B^{(\orbifold{O},\metric{g})}(x_\infty,3\rho)$. Given $k\in\nat$, let
\[
	y_k
	=
	\begin{cases}
		x_k,~&\text{if}~x_\infty\in\orbifold{O}^{\reg}~\text{and}\\
		x_k^{\sing},~&\text{if}~x_\infty\in\orbifold{O}^{\sing}.
	\end{cases}
\]
Up to discarding a finite number of indices again, we can use Remark \ref{ExponentialMap:Remark} to guarantee that we can fix $\rho'>0$ such that $\B^{(T_{y_k}\orbifold{O},\metric{g}_{y_k})}(0,\rho')\subset\Omega^{(\orbifold{O},\metric{g})}$ for any $k\in\nat$. Given $k\in\nat$,
\begin{itemize}
\item
Lemma \ref{Concentration:Lem:Ekeland} allows us to suppose that
\[
	\mathrm{d}_{u_k}J_{\epsilon_k}(v)^2
	<
	\frac{\delta_k}{\epsilon_k^n}
	\int_{\orbifold{O}}
		\epsilon^2\metric{g}(\grad^{(\orbifold{O},\metric{g})}(v),\grad^{(\orbifold{O},\metric{g})}(v))
		+
		v^2
	\mathrm{d}\mu^{(\orbifold{O},\metric{g})}.
\]
for any $v\in\T_{u_k}\mathcal{N}_{\epsilon_k}\subset H^1(\orbifold{O},\metric{g})$;
\item
we let $(\widetilde{U_k}:=B^{\real^n}(0,\rho'),G_k,\phi_k)$ be a normal chart around $y_k$ on $(\orbifold{O},\metric{g})$ and
\item
we define $w_k\in H_0^1(B^{\real^n}(0,\rho'/\epsilon_k))$ as $w_k(\widetilde{x})=\chi_{\rho'}(\epsilon_k|\widetilde{x}|)u_k(\phi_k(\epsilon_k\widetilde{x}))$ for a.e. $\widetilde{x}\in B^{\real^n}(0,\rho'/\epsilon_k)$.
\end{itemize}

In fact,
\begin{multline}\label{Concentration:Benci5.6}
	\text{there exists}~w\in H^1(\real^n)~\text{such that}~w_k\rightharpoonup w~\text{in}~H^1(\real^n)~\text{and}
	\\
	w_k\to w~\text{in}~L^p_{\loc}(\real^n)~\text{as}~k\to\infty;
\end{multline}
\begin{equation}\label{Concentration:Benci5.7}
	\text{the function}~w\in H^1(\real^n)~\text{is a weak solution to}~
	\begin{cases}
		-\Delta w+ w=w|w|^{p-2}~\text{and}\\
		w\geq 0;
	\end{cases}
\end{equation}
and
\begin{equation}\label{Concentration:Benci5.8}
	E(w)=m(E)
\end{equation}
The results \eqref{Concentration:Benci5.6}--\eqref{Concentration:Benci5.8} are the same as \cite[Lemmas 5.6--5.8]{Benci2007}, so we refer the reader to \cite[Section 7]{Benci2007} for their proofs. Furthermore, the set $\{\phi_k\}_{k\in\nat}$ is composed of chart maps of normal charts on $(\orbifold{O},\metric{g})$ and $\epsilon_k\to 0$, $x_k\to x_\infty$ as $k\to\infty$, so
\begin{multline}\label{Concentration:Eqn:Aux}
	\text{given}~\sigma\in]0,1[,~\text{there exists}~k_\sigma\in\nat~\text{such that if}~k\geq k_\sigma,~\text{then}
	\\
	1-\sigma\leq|\det\widetilde{\metric{g}_k}(\epsilon_kz)|^{1/2}\leq 1+\sigma~\text{for every}~z\in\B^{\real^n}(0,\rho').
\end{multline}

Let us prove that $x_\infty\in\B^{(\orbifold{O},\metric{g})}(\Z^{\orbifold{O}},\rho)$. Suppose otherwise. It follows from \eqref{Concentration:Eqn:Aux} that given $\sigma\in]0,1[$, we have
\begin{align*}
	\|u_k^+\|^p_{L^p(\orbifold{O},\metric{g})}
	&\geq
	\int_{\B^{(\orbifold{O},\metric{g})}(x_k,\rho')}
		\chi_{\rho'}(\dist^{(\orbifold{O},\metric{g})}(y,x_k))^pu_k^+(y)^p
	\mathrm{d}\mu^{(\orbifold{O},\metric{g})}(y);\\
	&\geq
	\frac{1}{|G_k|}
	\int_{\B^{\real^n}(0,\rho')}
		|\det\widetilde{\metric{g}_k}(\widetilde{x})|^{1/2}
		\chi_{\rho'}(|\widetilde{x}|)^pu_k^+(\phi_k(\widetilde{x}))^p
	\mathrm{d}\widetilde{x};\\
	&\geq
	\frac{1-\sigma}{|G_k|}\epsilon_k^n
	\int_{\B^{\real^n}(0,\rho'/\epsilon_k)}
		\chi_{\rho'}(\epsilon_k|\widetilde{x}|)^pu_k^+(\phi_k(\epsilon_k\widetilde{x}))^p
	\mathrm{d}\widetilde{x};\\
	&\geq
	\frac{1-\sigma}{|G_k|}\epsilon_k^n\|w_k^+\|_{L^p(\real^n)}^p
	\geq
	\frac{1-\sigma}{\zeta^{\orbifold{O}}-1}\epsilon_k^n\|w_k^+\|_{L^p(\real^n)}^p
\end{align*}
whenever $k\in\nat$ is sufficiently large. Due to \eqref{Concentration:Benci5.6} and \eqref{Concentration:Benci5.8}, we obtain
\[
	\liminf_{k\to\infty}\|u_k^+\|^p_{L^p(\orbifold{O},\metric{g})}
	\geq
	\frac{1}{\zeta^{\orbifold{O}}-1}\frac{2p}{p-2}m(E),
\]
which contradicts the fact that $J_{\epsilon_k}(u_k)<(\zeta^\orbifold{O})^{-1}m(E)+\delta_k$ for every $k\in\nat$ because $\delta_k\to 0$ as $k\to\infty$.

As $\B^{(\orbifold{O},\metric{g})}(\Z^{\orbifold{O}},\rho)$ is an open subset of $\orbifold{O}$, it follows from the claim in the previous paragraph that, up to discarding a finite number of indices, $\{x_k\}_{k\in\nat}\subset\B^{(\orbifold{O},\metric{g})}(\Z^{\orbifold{O}},\rho)$. At this point, we should argue as in the last part of the proof of \cite[Proposition 5.5]{Benci2007} in order to finish (more precisely, see \cite[p. 480]{Benci2007}).
\end{proof}

Let us finally prove Theorem \ref{Concentration:Thm}.
\begin{proof}[Proof of Theorem \ref{Concentration:Thm}]
Due to Lemma \ref{TheInjection:Lem:Energy}, we know that $\limsup_{\epsilon\to 0^+} m(J_\epsilon)\leq m(E)/\zeta^\orbifold{O}$, so the limit $m(J_\epsilon)\to m(E)/\zeta^\orbifold{O}$ as $\epsilon\to 0^+$ is a corollary of Lemma \ref{Concentration:Lemma:2}. The result then follows from this limit and Remark \ref{TheInjection:Rmk:n_epsilon}, which implies $J_\epsilon(u)=(p-2)\|u^+\|^p_{L^p(\orbifold{O},\metric{g})}/(2p)$ for every $\epsilon\in]0,1[$ and $u\in\mathcal{N}_\epsilon$.
\end{proof}

\subsection*{Extending the Riemannian center of mass}

The notion of the Riemannian center of mass on Riemannian manifolds was extended by Petean in \cite[Section 5]{Petean2019} to also encompass functions which are sufficiently concentrated, but not necessarily supported, on geodesic balls with a small radius.

Let us to show that Petean's constructions also hold in the context of Riemannian orbifolds. First of all, we set
\[
	L^{1,r}(\orbifold{O},\metric{g})=\{
		u\in L^1(\orbifold{O},\metric{g}):\supp u\subset\B^{(\orbifold{O},\metric{g})}(x,r)~\text{for a certain}~x\in\orbifold{O}
	\}
\]
for any $r>0$. Given $u\in L^{1,\rho}(\orbifold{O},\metric{g})\setminus\{0\}$, consider the continuous function given by
\[
	P_u(x):=\frac{1}{\|u\|_{L^1(\orbifold{O},\metric{g})}}\int_{\orbifold{O}}\dist^{(\orbifold{O},\metric{g})}(x,y)^2|u(y)|\mathrm{d}\mu^{(\orbifold{O},\metric{g})}(y)
\]
for every $x\in\orbifold{O}$. It is easy to check that
\begin{equation}\label{c_epsilon:Equation:Mapping_P}
	L^{1,\rho}(\orbifold{O},\metric{g})\setminus\{0\}\ni u\mapsto P_u \in C(\orbifold{O},\real)
\end{equation}
is continuous and a simple argument by contradiction proves the following lemma.
\begin{lem}
If $K$ is a compact metric space and
\[
	\mathcal{A}:=\{f\in C(K,\real):f~\text{has a unique minimum point}\},
\]
then the mapping that takes a function in $\mathcal{A}$ to its unique minimum point in $K$ is continuous.
\end{lem}

Due to Proposition \ref{CenterOfMass:Proposition} and the hypotheses on $\rho$ (see Definition \ref{ConvenientRadius:Definition:Second}), we can define $\cm\colon L^{1,\rho}(\orbifold{O},\metric{g})\setminus\{0\}\to\orbifold{O}$ as the mapping that associates $u\in L^{1,\rho}(\orbifold{O},\metric{g})\setminus\{0\}$ to the unique minimum point of $P_u$. In fact, $\cm\colon L^{1,\rho}(\orbifold{O},\metric{g})\setminus\{0\}\to\orbifold{O}$ is continuous as a composition of the continuous mapping \eqref{c_epsilon:Equation:Mapping_P} with the mapping in the previous lemma.

Given $u\in L^1(\orbifold{O},\metric{g})\setminus\{0\}$, we define its \emph{concentration function} as the continuous function given by
\[
	\CF_u(x):=\frac{1}{\|u\|_{L^1(\orbifold{O},\metric{g})}}\int_{\B^{(\orbifold{O},\metric{g})}(x,\rho)} |u|\mathrm{d}\mu^{(\orbifold{O},\metric{g})}
\]
for any $x\in\orbifold{O}$ and we define its \emph{concentration coefficient} as $\CC(u)=\max\CF_u$. If $1/2<\eta<1$ and $\CC(u)>\eta$, then we define $\Psi_{\eta,u}\in L^1(\orbifold{O},\metric{g})$ as
\[
	\Psi_{\eta,u}(x)=\left[
		1-\chi_\eta(\CF_u(x))
	\right]u(x)
\]
for $\mu^{(\orbifold{O},\metric{g})}$-a.e. $x\in\orbifold{O}$. The following result shows that $\Psi_{\eta,u}$ is supported in a small metric ball whenever $u$ is sufficiently concentrated around a point.

\begin{lem}[{\cite[Lemma 5.1]{Petean2019}}]
If $1/2<\eta<1$ and $\CF_u(x)=\CC(u)>\eta$, then $\supp \Psi_{\eta,u}\subset\B^{(\orbifold{O},\metric{g})}(x,2\rho)$.
\end{lem}

It follows from the previous lemma that if $1/2<\eta<1$, then
\[
	\CC^{-1}(]\eta,1])\ni u\mapsto \Psi_{\eta,u}\in L^{1,2\rho}(\orbifold{O},\metric{g})
\]
is continuous and we finally obtain an analogue to \cite[Theorem 5.2]{Petean2019}, which may be proved similarly.

\begin{thm}\label{c_epsilon:Theorem}
Given $\eta\in]1/2,1[$, the mapping
\[
	\CC^{-1}(]\eta,1])\ni u
	\mapsto
	\Cm_\eta(u):=\cm(\Psi_{\eta,u})\in\orbifold{O}
\]
is continuous and if $\CF_u(x)=\CC(u)>\eta$, then $\Cm_\eta(u)\in\B^{(\orbifold{O},\metric{g})}(x,2\rho)$.
\end{thm}

\section{Proof of Theorem \ref{Introduction:Theorem:ToProve}}\label{ProofA}

Fix $\eta\in]1/2,1[$ and let $\epsilon_\eta\in]0,1]$, $\delta_\eta\in]0,\infty[$ be such that the conclusion of Theorem \ref{Concentration:Thm} holds. Due to Lemma \ref{TheInjection:Lem:Energy}, we can fix $\widetilde{\epsilon}_\eta\in]0,\epsilon_\eta]$ such that whenever $0<\epsilon<\widetilde{\epsilon}_\eta$, we have a map $\Z^{\orbifold{O}}\ni x\mapsto i_\epsilon(x)\in\Lambda_{\epsilon,(\zeta^{\orbifold{O}})^{-1}m(E)+\delta_\eta}$. Suppose that $(\epsilon,\delta)\in]0,\widetilde{\epsilon}_\eta[\times]0,\delta_\eta[$. Theorem \ref{Concentration:Thm} shows that if $u\in\Lambda_{\epsilon,(\zeta^{\orbifold{O}})^{-1}m(E)+\delta}$, then $\CF_{(u^+)^p}(x)>\eta$ for a certain $x\in\B^{(\orbifold{O},\metric{g})}(\Z^{\orbifold{O}},\rho)$. We can therefore use Theorem \ref{c_epsilon:Theorem} to set
\[
	\Lambda_{\epsilon,(\zeta^{\orbifold{O}})^{-1}m(E)+\delta}\ni u\mapsto c_\epsilon(u):=\Cm_\eta((u^+)^p)\in \B^{(\orbifold{O},\metric{g})}(\Z^{\orbifold{O}},3\rho).
\]
Lemma \ref{ConvenientRadius:Lemma:Homeomorphism} assures that
\[
	\{
		w\in \left.N\right|_{\Z^{\orbifold{O}}}:\metric{g}(w,w)<(3\rho)^2	
	\}\ni v\mapsto\psi(v):=\exp^{(\orbifold{O},\metric{g})}(v)\in\B^{(\orbifold{O},\metric{g})}(\Z^{\orbifold{O}},3\rho)
\]
is a homeomorphism, so the function $H\colon\orbifold{O}\times[0,1]\to\orbifold{O}$ given by
\[
	H(x,t)=\psi\left(
		(1-t)\psi^{-1}(c_\epsilon\circ i_\epsilon(x))
	\right)
\]
furnishes a homotopy from $c_\epsilon\circ i_\epsilon$ to $\id_{\Z^{\orbifold{O}}}$ and the result follows from the photography method as sketched in the introduction.

\qed

\appendix
\section{Length spaces}\label{LengthSpaces}
Let $(X,\dist)$ be a \emph{length space}, i.e., $(X,\dist)$ is a metric space and
\begin{multline*}
	\dist(x,y)
	=
	\inf\{
		\length(\gamma)
		\mid
		\\
		\mid
		\gamma\colon [0,1]\to (X,\dist)~\text{is continuous}, \gamma(0)=x~\text{and}~\gamma(1)=y
	\}
\end{multline*}
for any $x,y\in X$, where given a continuous $\gamma\colon [a,b]\to (X,\dist)$,
\[
	\length(\gamma):=
	\sup\left\{
		\sum_{i=0}^n \dist(\gamma(t_i),\gamma(t_{i+1}))
		\mid
		n\in\nat_0~\text{and}~
		a=t_0\leq\ldots\leq t_{n+1}=b
	\right\}.
\]

If $\gamma\colon [a,b]\to (X,\dist)$ is continuous, then we say that $\gamma$ is a \emph{minimizing geodesic} when
\[
	\dist(\gamma(t),\gamma(s))=\frac{|t-s|}{|b-a|}\length(\gamma)
\]
for any $t,s\in[a,b]$ and we call $\gamma$ a \emph{geodesic} when there exists $\epsilon>0$ such that if $a\leq t\leq s\leq \min(b,t+\epsilon)$, then $\left.\gamma\right|_{[t,s]}$ is a minimizing geodesic.

An $F\colon (X,\dist)\to\real\cup\{\infty\}$ is said to be \emph{convex} (resp., \emph{strictly convex}) if $F\circ\gamma$ is convex (resp., strictly convex) whenever $\gamma\colon [a,b]\to (X,\dist)$ is a geodesic (resp., a non-constant geodesic). In fact, strict convexity is related to the uniqueness of minimum points.

\begin{prop}[{\cite[Lemma 3.1.1]{Jost1997}}]\label{LengthSpaces:Proposition:StrictlyConvex}
Suppose that $F\colon (X,\dist)\to\real\cup\{\infty\}$ is strictly convex and given $x,y\in X$, there exists a minimizing geodesic that links $x$ to $y$. We conclude that $F$ admits at most one minimum point in $X$.
\end{prop}

\sloppy
\printbibliography
\end{document}